\documentclass[platex,11pt]{article}
\usepackage{color}
\usepackage{graphicx}
\usepackage{hyperref}
\usepackage{setspace}
\usepackage{amsfonts}
\usepackage{mathtools}
\usepackage{wrapfig}

\usepackage{amsmath,amssymb,amsthm}
\usepackage[top=30truemm,bottom=30truemm,left=25truemm,right=25truemm]{geometry}
\usepackage{bm}
\usepackage{ascmac}
\usepackage[titletoc,toc,title]{appendix}
\AtBeginDvi{}
\usepackage{times}

\usepackage{titlesec}

\def\@maketitle{
  \newpage
  \null
  \vskip 2em
  \begin{center}
  \let \footnote \thanks
    {\large\bf \@title \par}
    \vskip 1.5em
    {
      \lineskip .5em
      \begin{tabular}[t]{c}
        \@author
      \end{tabular}\par}
    \vskip 1em
  \end{center}
  \par
}

\makeatletter
\renewcommand \thesection {\@arabic\c@section}

\renewcommand\section{\@startsection {section}{1}{\z@}
                                   {-1.5ex \@plus -1ex \@minus -.2ex}
                                   {2.3ex \@plus.2ex}
                                   {\Large\normalfont\bfseries}}
\renewcommand\subsection{\@startsection{subsection}{2}{\z@}
                                     {-3.25ex\@plus -1ex \@minus -.2ex}
                                     {1.5ex \@plus .2ex}
                                     {\normalfont\normalsize\bfseries}}
\renewcommand\subsubsection{\@startsection{subsubsection}{3}{\z@}
                                     {-3.25ex\@plus -1ex \@minus -.2ex}
                                     {1.5ex \@plus .2ex}
                                     {\normalfont\normalsize\it}}

\renewcommand\appendix{\par
 \setcounter{section}{0}
 \setcounter{subsection}{0}
 \renewcommand\thesection{\@Alph\c@section}
}
\makeatother

\newcommand{\BM}{Brownian motion}

\newcommand{\BS}{Black-Scholes}
\newcommand{\SDE}{stochastic differential equation}
\newcommand{\CEV}{constant elasticity of variance}
\newcommand{\MC}{Malliavin calculus}

\newcommand{\RV}{random variable}

\newcommand{\WRT}{with respect to}

\newtheorem{theorem*}{Theorem}[section]
\newtheorem{definition*}{Definition}[section]
\newtheorem{lemma*}{Lemma}[section]
\newtheorem{cor*}{Corollary}[section]
\newtheorem{property*}{property}
\newtheorem{remark*}{remark}

\makeatletter
\@addtoreset{equation}{section}

\makeatletter

\title{\huge Malliavin Differentiability of CEV-Type Heston Model}

\author{SHOTA TSUMURAI\thanks{Graduate School of Mathematics, Keio University, 3-14-1, Hiyoshi, Kohoku-ku, Yokohama-shi, Kanagawa, 223-8522, Japan, Email: syota.tsumurai@keio.jp}}
\date{January 19,2020}

\begin{document}

\pagenumbering{arabic}

\maketitle
\begin{center}
{Abstract}\\[0.2cm]
\end{center}
It is well known that Malliavin calculus can be applied to a stochastic differential equation with Lipschitz continuous coefficients in order to clarify the existence and the smoothness of the solution. In this paper, we apply Malliavin calculus to the CEV-type Heston model whose diffusion coefficient is non-Lipschitz continuous and prove the Malliavin differentiability of the model.

\renewcommand{\contentsname}{Contents}

\renewcommand{\refname}{References}

\tableofcontents

\section{Introduction}

Malliavin calculus is the infinite-dimensional differential calculus on the Wiener space in order to give a probabilistic proof of H\"{o}lmander's theorem. It has been developed as a tool in mathematical finance. In 1999, Founi\'{e} et al.\,\cite{F1} gave a new method for more efficient computation of Greeks which represent sensitivities of the derivative price to changes in parameters of a model under consideration, by using the integration by parts formula related to Malliavin calculus. Following their works, more general and efficient application to computation of Greeks have been introduced by many authors (see \cite{B1}, \cite{B2}, \cite{F2}). They often considered this method for tractable models typified by the Black-Scholes model.

In the Black-Scholes model, an underlying asset $S_t$ is assumed to follow the {\SDE} $dS_t=rS_t\,dt+{\sigma}S_t\,dW_t$, where $r$ and ${\sigma}$ respectively imply the risk free interest rate and the volatility. The Black-Scholes model seems standard in business. The reason is that this model has the analytic solution for famous options, so it is fast to calculate prices of derivatives and risk parameters (Greeks) and easy to evaluate a lot of deals and the whole portfolios and to manage the risk. However, the {\BS} model has a defect that this model assumes that volatility is a constant.

In the actual financial market, it is observed that volatility fluctuates. However, the {\BS} model does not suppose the prospective fluctuation of volatility, so when we use the model there is a problem that we would underestimate prices of options. Hence, more accurate models have been developed. One of the models is the stochastic volatility model. One of merits to consider this model is that even if prices of derivatives such as the European options are not given for any strike and maturity, we can grasp the volatility term structure. In particular, the Heston model, which is introduced in \cite{H}, is one of the most popular stochastic volatility models. This model assumes that the underlying asset ${S_t}$ and the volatility $\nu_t$ follow the {\SDE}s
\begin{align}
dS_t&=S_t (r\,dt+\sqrt{\nu_t}\,dB_t),\label{1}\\
d{\nu}_t &=\kappa(\mu-{\nu}_t)\,dt+ {\theta}{{\nu}_t}^{\frac{1}{2}}\,dW_t,\label{2}
\end{align}
where $B_t$ and $W_t$ denote correlated {\BM}s. In the equation \eqref{2}, $\kappa$, $\mu$ and $\theta$ imply respectively the rate of mean reversion (percentage drift), the long-run mean (equilibrium level) and the volatility of volatility. This volatility model is called the Cox-Ingersoll-Ross model and more complicated than the {\BS} model. We have not got the analytic solution yet.

However, even this model can not grasp fluctuation of volatility accurately. In 2006 (see \cite{AP}), Andersen and Piterbarg generalized the Heston model. They extended the volatility process of \eqref{2} to
\begin{align}
d{\nu}_t = \kappa(\mu-{\nu}_t)\,dt+ {\theta}{{\nu}_t}^{\gamma}\,dW_t,\quad\gamma\in\left[\frac{1}{2}, 1\right].\label{3}
\end{align}
This model is called the {\CEV} model (we will often shorten this model as the CEV model). Naturally, in the case $\gamma\in\left(\frac{1}{2},1\right)$, the volatility model \eqref{3} is more complicated than the volatility model \eqref{2}.

Here, consider the European call option and let $\phi$ is a payoff function. Then we can estimate the option price by the following formula $V(x)=E[\,e^{-rT}\phi(S_T)]$. However, the computation of Greeks is much important in the risk-management. A Greek is given by $\frac{\partial{V(x)}}{\partial\alpha}$ where $\alpha$ is one of parameters needed to compute the price, such as the initial price, the risk free interest rate, the volatility and the maturity etc.. Most of financial institutions have calculated Greeks by using finite-difference methods but there are some demerits such that the results depend on the approximation parameters. More than anything, the methods need the assumption that the payoff function $\phi$ is differentiable. However, in business they often consider the payoff functions such as $\phi(x)=(x-K)_+$ or $\phi(x)={\bf{1}}_{\{x\ge{K}\}}$. Here we need Malliavin calculus. In 1999 Founi\'{e} et al.\,in \cite{F1} gave the new methods for Greeks. To come to the point, they calculated Greeks by the following formula $\frac{\partial{V(x)}}{\partial\alpha}=E[\,e^{-rT}\phi(S_T){\cdot}({\rm{weight}})]$. We can calculate this even if $\phi$ is polynomial growth. Instead, we need the Malliavin differentiability of $S_t$.

The solution $X_t$ satisfying the {\SDE} with Lipschitz continuous coefficients is known as Malliavin differentiable. Hence we can easily verify that the {\BS} model is Malliavin differentiable. However the diffusion coefficient $x^{\gamma},\,\gamma\in\left[\frac{1}{2},1\right)$ is neither differentiable at $x=0$ nor Lipschitz continuous and then we cannot find whether the CEV-type Heston model is Malliavin differentiable or not. In \cite{AE}, Alos and Ewald proved that the volatility process \eqref{2}, that is the case where $\gamma=\frac{1}{2}$ of \eqref{3}, was Malliavin differentiable and gave the explicit expression for the derivative. However, in the case $\gamma\in\left(\frac{1}{2},1\right)$, we can not simply prove the Malliavin differentiability in the exact same way.

In this paper we concentrate on the case $\gamma\in\left(\frac{1}{2},1\right)$, that is, we extend the results in \cite{AE} and give the explicit expression for the derivative. Moreover we consider the CEV-type Heston model and give the formula to compute Greeks.

\section{Summary of Malliavin Calculus}

We give the short introduction of Malliavin calculus on the Wiener space. For further details, refer to \cite{N}.

\subsection{Malliavin Derivative}
We consider a {\BM} $\{W(t,\omega)\}_{t\in[0,T]}$ (in the sequel, we often denote $W(t,\omega)$ by $W_t$) on a complete filtered probability space $(\Omega,\,\cal{F},\,\rm{P};\,({\cal{F}}_{t}))$ where $({\cal{F}}_{t})$ is the filtration generated by $W_t$, and the Hilbert space $H\coloneqq L^2([0,T])$. When fixing $\omega$, we can consider $\omega(t)\coloneqq{W(t,\omega)}\in{\bf{C}}([0,T])$. Then the It\^o integral of $h\in H$ is constructed as $\displaystyle\int_0^T h(t) \,dW(t,\omega)=\int_0^T h(t) \,d{\omega(t)}$ on ${\bf{C}}([0,T])$.
We denote by $C_p^{\infty}({{\bf{R}^{n}}})$ the set of infinitely continuously differentiable functions $f:\,{\bf{R}^{n}}\to{\bf{R}}$ such that $f$ and all its partial derivatives have polynomial growth. Let ${\cal{S}}$ be the space of smooth random variables expressed as
\begin{align}
F(\omega)=f(W(h_1),\ldots,W(h_n)),\label{21}
\end{align}
where $f\in{C}_p^{\infty}({{\bf{R}^{n}}})$ and $W(h)\displaystyle\coloneqq\int_0^T h(t) \,dW_t$ where $h_1,\,\ldots,\,h_n\in{H}$, $n\ge1$.
We denote by $C_0^{\infty}({{\bf{R}^{n}}})$ the set of infinitely continuously differentiable functions $f:\,{\bf{R}^{n}}\to{\bf{R}}$ such that $f$ has compact support. Moreover we denote by $C_b^{\infty}({{\bf{R}^{n}}})$ the set of infinitely continuously differentiable functions $f:\,{\bf{R}^{n}}\to{\bf{R}}$ such that $f$ and all of its partial derivatives are bounded. Denote by ${\cal{S}}_0$ and ${\cal{S}}_b$ respectively, the spaces of smooth random variables of the form \eqref{21} such that $f\in{C_0^{\infty}({{\bf{R}^{n}}})}$ and $f\in{C_b^{\infty}({{\bf{R}^{n}}})}$. We can find that ${\cal{S}}_0 \subset{\cal{S}}_b\subset{\cal{S}}$ and ${\cal{S}}_0$ is a linear subspace of and dense in $L^p(\Omega)$ for all $p>0$. We use the notation $\partial_i=\displaystyle\frac{\partial}{\partial{x_i}}$ in the sequal. We define the derivative operator $D$, so called the Malliavin derivative operator.

\begin{definition*}[Malliavin derivative]
The Malliavin derivative $D_tF$ of a smooth {\RV} expressed as \eqref{21} is defined as the $H$-valued {\RV} given by
\begin{align}
D_tF=\sum_{i=1}^n\partial_if(W(h_1),\,\ldots,\,W(h_n))h_i(t).
\end{align}
We sometimes omit to write the subscript $t$.
\end{definition*}
Since ${\cal{S}}$ is dense in $L^p(\Omega)$, we will define the Malliavin derivative of a general $F\in L^p(\Omega)$ by means of taking limits.
We will now prove that the Malliavin derivative operator $D:\,L^p(\Omega)\to L^p(\Omega;H)$ is closable.
Please refer to \cite{N} for proves of the following results.
\begin{lemma*}
We have \,\,$E[\,G\langle{DF,h}\rangle_H]=-E[\,F\langle{DG,h}\rangle_H]+E[\,FGW(h)$, for $F,G\in{\cal{S}}$ and $h\in{H}$.
\end{lemma*}

\begin{lemma*}
For any $p\ge1$, the Malliavin derivative operator $D:\,{L^p(\Omega)}\to{L}^p(\Omega;H)$ is closable.
\end{lemma*}

For any $p\ge1$, we denote by ${\bf{D}}^{1,p}$ the domain of $D$ in $L^p(\Omega)$ and then it is the closure of ${\cal{S}}$ by the norm
\begin{align}
\|{F}\|_{1,p}=\left\{E[\,|F|^p]+E[\,\|DF\|_H^p]\right\}^{\frac{1}{p}}=\left\{E[\,|F|^p]+E\left[\left(\int_0^T |D_tF|^2 \,dt\right)^{\frac{p}{2}}\right]\right\}^{\frac{1}{p}}.
\end{align}
Note that ${\bf{D}}^{1,2}$ is a Hilbert space with the scalar product $\langle{F,G}\rangle_{1,2}=E[\,FG]+E[\,\langle{DF,DG}\rangle_H]$. Moreover, the Malliavin derivative $\{D_tF\}_{t\in[0,T]}$ is regarded as a stochastic process defined almost surely with the measure $P\times{u}$ where $u$ is a Lebesgue measure in $[0,T]$. Indeed, we can observe
\begin{align}
\|DF\|^2_{L^2(\Omega;H)}=E\left[\int_0^T(D_tF)^2\,dt\right]=\int_0^T E[(D_tF)^2]\,dt=\|D_{\cdot}F\|^2_{L^2(\Omega\times [0,T])}.
\end{align}
The following result will become a very important tool.
\begin{lemma*}
  Suppose that a sequence $\{F_n:F_n\in{\bf{D}}^{1,2}, \sup_{n}E[\,\|DF_n\|^2_H]<\infty\}$ converges to $F$ in $L^2(\Omega)$.
  Then $F$ belongs to ${\bf{D}}^{1,2}$ and the sequence $\{DF_n\}$ converges to $DF$ in the weak topology of ${L}^2(\Omega;H)$.
\end{lemma*}

Similarly, we define the $k$-th Malliavin derivative of $F$, $\{D^k_{t_1,\ldots,t_k}F, \,t_i\in[0,T]\}$, as a $\Omega\times[0,T]^k$-measurable stochastic process defined $P\times{u}^k$-almost surely and the operator $D^k$ is closable from ${\cal{S}}\to L^p(\Omega;H^k)$ for any $p\ge1$ and $k\ge1$. As with the Malliavin derivative $D$, from the closability of $D^k$, we can define the domain ${\bf{D}}^{k,p}$ of the operator $D^k$ in $L^p(\Omega)$ as the completion of ${\cal{S}}$ with the norm
\begin{align}
\|F\|_{k,p}=\left\{E[|F|^p]+\sum_{i=1}^{k}E[\|D^kF\|^p_{H^{\otimes i}}]\right\}^{\frac{1}{p}}.
\end{align}
Moreover we define ${\bf{D}}^{1,\infty}$ as ${\bf{D}}^{1,\infty}\displaystyle\coloneqq\bigcap_{p\in{\bf{N}}}{\bf{D}}^{1,p}$.
We will now prove the chain rule and refer to the \rm{\cite[Proposition 1.2.4]{N}} for details.
\begin{lemma*}
  For $p\ge1$, let $F=(F_1,\,\ldots,\,F_n)\in{\bf{D}}^{1,p}$ and $\psi:\,{\bf{R}}^n\to{\bf{R}}$ be a Lipschitz function with bounded partial derivatives, and then we have $\psi(F)\in{\bf{D}}^{1,p}$ and
  \begin{align}
    D_t\psi(F)=\sum_{i=1}^n \partial_i \psi(F)D_tF_i.
    \end{align}
  \end{lemma*}

\subsection{Skorohod Integral}
For $p,q>1$ satisfing $1/p+1/q=1$, the adjoint $D^*$ of the operator $D$ which is closable and has the domain on $L^p(\Omega)$ should be closable but with the domain contained in $L^q(\Omega)$. Focus on the case $p=q=2$. We can define the divergence operator $\delta=D^{*}$ so called the Scorohod integral which is the adjoint of the operator $D$ such as
\begin{align}
\delta\,:\,{L}^2(\Omega;H)\cong L^2(\Omega\times [0,T])\to {L}^2(\Omega).
\end{align}
\begin{definition*}[Skorohod integral]
Let $u\in{L}^2(\Omega;H)$. If for all $F\in{{\bf{D}}}^{1,2}$, we can have \begin{align}
|E[\,\langle{DF,u}\rangle_H]|\le{c}\|F\|_{L^2(\Omega)},
\end{align}
where $c$ is some constant depending on $u$, then $u$ is called to belong to the domain $Dom(\delta)$. Moreover if $u\in{Dom(\delta)}$, then we have that $\delta(u)$ belongs to ${L}^2(\Omega)$ and the duality relation $E[\,F\delta(u)]=E[\,{\langle{DF,u}\rangle}_H]$, for all $F\in{\bf{D}}^{1,2}$.
\end{definition*}
We can get the following results.
\begin{lemma*}
Let $F\in{\bf{D}}^{1,2}$ and $u\in{Dom(\delta)}$ satisfy $Fu\in{L}^2(\Omega;H)$. And then we have that $Fu$ belongs to $Dom(\delta)$ and $\delta(Fu)=F\delta(u)-{\langle{DF,u}\rangle}_H$.
\end{lemma*}

\begin{lemma*}
  Let $u\in{L}^2(\Omega;H)$ be an ${\cal{F}}_t$-adapted stochastic process then $u\in{Dom(\delta)}$ and $\delta(u)=\int_0^T u_t\,dW_t$.
\end{lemma*}

We give one of famous properties of $\delta$. The following property implies the relationship between the Malliavin derivative and the Skorohod integral. Denote by ${\bf{D}}^{1,2}(H)$ the class of processes $u\in L^2(\Omega;H)\cong (\Omega\times[0,T])$ such that $u(t)\in {\bf{D}}^{1,2}$ for almost all $t$ and there exists a measurable version of the two variable processes $D_su_t$ satisfying $E\left[\int_0^T \int_0^T (D_su_t)^2 \,\lambda(ds)\,\lambda(dt)\right]<\infty$.
\begin{lemma*}
  Let $u\in{\bf{D}}^{1,2}(H)$ satisfy that $D_r{u_t}\in{Dom(\delta)}$ and that $\delta(D_ru_t)\in{L}^2(\Omega;H)$. We have then that $\delta(u)$ belongs to ${\bf{D}}^{1,2}$ and
  \begin{align}
  D_t(\delta(u))=u(t)+\delta(D_tu).
  \end{align}

  \end{lemma*}
  The following result is applied to calculate Greeks. For further details, refer to \cite[Chapter 6]{N}.
  \begin{lemma*}
    Let $F,G\in{\bf{D}}^{1,2}$. Suppose that an random variable $u(t,\cdot)\in{H}$ satisfy $\langle{DF,u}\rangle_H\neq0$ a.s. and $Gu(\langle{DF,u}\rangle_H)^{-1}\in{Dom(\delta)}$. For any continuously differentiable function $f$ with bounded derivatives, we have $E[\,f'(F)G]=E[\,f(F)H(F,G)]$ where $H(F,G)=\delta(Gu(\langle{DF,u}\rangle_H)^{-1})$.
  \end{lemma*}

\subsection{Malliavin Calculus for Stochastic Differential Equations}

Consider $T>0$ and $\Omega=C_0([0,T];{\bf{R}}^m)$. Let $\{W_t\}_{t\in[0,T]}$ be the $m$-dimensional {\BM} on filtered probability space $(\Omega,\,{\cal{F}},\,P;{\cal{F}}_t)$ where $P$ is the $n$-dimensional Wiener measure and ${\cal{F}}$ is the completion of the $\sigma$-field of $\Omega$ with $P$. And then $H=L^2([0,T];{\bf{R}}^m)$ is the underlying Hilbert space.
We consider the solution $\{X_t\}_{t\in[0,T]}$ of the following $n$-dimensional {\SDE} for all $i=1,\ldots,n$
\begin{align}
  dX_t^i=b^i(X_t)\,dt+\sum_{j=1}^m\sigma_j^i(X_t)dW^j_t,{\quad}X_0^i=x^i,\label{22}
  \end{align}
where $b:{\bf{R}}^n\to{\bf{R}}^n$ and $\sigma_j:{\bf{R}}^n\to{\bf{R}}^m$ satisfy the following : there is a positive constant $K<\infty$ such that
\begin{align}
&|b(x)-b(y)|+|\sigma(x)-\sigma(y)| \le K|x-y|,\quad{\rm{for\,\,all}}\,\,x,\,y\in{\bf{R}}^n,\\
&|b(x)|+|\sigma(x)|\le{K}(1+|x|),\quad{\rm{for\,\,all}}\,\,x\in{\bf{R}}^n.
\end{align}
Here $\sigma_j$ is the columns of the matrix $\sigma=(\sigma_j^i)$. We can have the following result related to the uniqueness and refer to \cite[Lemma 2.2.1]{N} for the detail.
\begin{theorem*}
  There is a unique $n$-dimensional, continuous and ${\cal{F}}_t$-adapted stochastic process $\{X_t\}_{t\in[0,T]}$ satisfying the {\SDE} \eqref{22} with $E\Big[\sup_{0\le{t}\le{T}}|X(t)|^p\Big]<\infty$, for all $p\ge2$.
\end{theorem*}
In the case the coefficients are Lipschitz, the solution $X^i_t$ belongs to ${\bf{D}}^{1,\infty}$.
\begin{theorem*}\label{thm22}
 Assume that coefficients are Lipschitz continuous of the {\SDE} \eqref{22}. Then the solution $X_t^i$ belongs to ${\bf{D}}^{1,\infty}$ for all $t\in[0,T]$ and $i=1,\ldots,n$ and satisfies
 \begin{align}
 \sup_{0\le{r}\le{t}}E\Big[\sup_{r\le{s}\le{T}}|D_r^jX_s^i|^p\Big]<\infty.
 \end{align}
 Moreover the derivative $D_r^jX_t^i$ satisfies the following
 \begin{align}
  D_r^jX_t^i=\sigma_j^i(X_r)+\sum_{k=1}^n\sum_{l=1}^m\int_r^t \partial_k\sigma_l^i(X_s)D_r^jX_s^k\,dW_s^l+\sum_{k=1}^n\int_r^t \partial_kb^i(X_s)D_r^jX^k_s\,ds,
  \end{align}
  for $r\le{t}$\, \rm{a.e.}, and $D_r^jX_t^i=0$ for $r>t$ \rm{a.e.}. Here $D^j$ denotes the Malliavin derivative for $W^j$.
  \end{theorem*}

Let $X_t$ be the solution of the following {\SDE}
\begin{align}
dX_t=b(X_t)\,dt + \sigma(X_t)\,dW_t,{\quad}X_0=x,
\end{align}
where $W_t$ denotes a $1$-dimensional {\BM}. Assume that $X_t\in{\bf{D}}^{1,2}$. We let $Y_t$ be the first variation of $X_t$, that is, $Y_t=\frac{\partial{X}_t}{\partial{x}}$. We can easily have that $Y_t$ satisfies the folloing
\begin{align}
dY_t=b'(X_t)Y_t\,dt+\sigma'(X_t)Y_t\,dW_t,{\quad}Y_0=1.
\end{align}
Considering this as a {\SDE} for $Y_t$, we can have the following solution
\begin{align}
Y_t=\exp\left\{\int_0^t(b'(X_s)-\frac{1}{2}(\sigma'(X_s))^2)\,ds+\int_0^t\sigma'(X_s)\,dW_s\right\}.
\end{align}
The following results will also be useful to calculate Greeks later.
\begin{lemma*}
Under the above conditions, we can have $Y_t=D_sX_t{\sigma}^{-1}(X_s)Y_s\cdot {\bf{1}}_{\{s\le{t}\}}$.
\end{lemma*}

Let $\{a(t)\}_{t\in[0,T]}$ be a continuous function in $H$ such that $\displaystyle\int_0^Ta(t)\,dt=1$.
\begin{lemma*}
  Under the above conditions, we can have $Y_T=\int_0^T a(t)D_tX_T\sigma^{-1}(X_t)Y_t\,dt$.
  \end{lemma*}

\begin{theorem*}
For any $\psi:{\bf{R}}\to{\bf{R}}$ of polynomial growth, we have $\frac{\partial}{\partial{x}}E[\,\psi(X_T)]=E[\,\psi(X_T)\pi]$ where $\pi=\displaystyle\int_0^T a(t)\sigma^{-1}(X_t)Y_t\,dW_t$.
\end{theorem*}

For the more general case, the same result is proved as below. Let $X_t$ denote the solution of the following $n$-dimensional {\SDE} just like as \eqref{22}
\begin{align}
dX_t=b(X_t)\,dt+\sigma(X_t)\,dW_t,{\quad}X_0=x,
\end{align}
where $W_t$ denotes $m$-dimensional {\BM}. For the sake of simplification, we assume that $n=m$.
\begin{theorem*}
Suppose that the diffusion coefficient $\sigma$ is invertible and that $E\left[\displaystyle\int_0^T|\sigma^{-1}(X_t)Y_t|^{2+\epsilon}\,dt\right]<\infty$, for some $\epsilon>0$, where $Y$ denotes the first variation process, that is, $Y^{ji}_t=\partial_{i}X_t^j$. Let $G\in{\bf{D}}^{1,\infty}$ be a {\RV} which does not depend on the initial condition $x$. Then for all measurable function $\phi$ with polynomial growth we have\,\,$\partial_iE[\,\phi(X_T)G]=E[\,\phi(X_T)\pi_i(G)$, where $a(t)$ is an ${\cal{F}}_t$-adapted process satisfying $\int_0^T a(t)\,dt=1$,
\begin{align}
  \pi_i(G)=&\sum_{k=1}^n\delta^k(Ga(t)(\sigma^{-1}(X_t)Y_t)^{ki})\notag\\
  =&\sum_{k=1}^n\left(G\int_0^T a(t)(\sigma^{-1}(X_t)Y_t)^{ki}dW^k_t-\int_0^T D^k_tGa(t)(\sigma^{-1}(X_t)Y_t)^{ki}\,ds\right),
  \end{align}
and $\delta^k$ denotes the adjoint to the Malliavin derivative {\WRT} a {\BM} $W_t^k$.
\end{theorem*}

The following theorem introduced in \cite{DGR} is useful. From now on, we will now denote by $\partial_t$ \,the once derivative with respect to $t$, by $\partial_x$\, the once derivative with respect to $x$ and by $\partial_{xx}$\, the second derivative with respect to $x$.

\begin{theorem*}
  Consider a stochastic process $X_t$ satisfying the $1$-dimensional {\SDE}
  \begin{align}
  dX_t=\mu(t,X_t)\,dt+\sigma(t,X_t)\,dW_t,
  \end{align}
  where $W_t$ denotes a {\BM} and the coefficients $\mu(t,x)\in{\bf{C}}^1([0,T]\times{\bf{R}})$ and $\sigma(t,x)\in{\bf{C}}^2([0,T]\times{\bf{R}})$ satisfy the linear growth condition and the Lipschitz condition. Moreover, we assume that $\sigma$ is positive and bounded away from 0, and that $\mu(t,0)$ and $\sigma(t,0)$ are bounded for all $t\in[0,T]$. Then $X_t$ belongs to ${\bf{D}}^{1,2}$ and the derivative is given by
  \begin{align}
    D_rX_t=\sigma(t,X_t)\exp\left\{\int_r^t \left(\partial_x{\mu}-\frac{\mu\partial_x{\sigma}}{\sigma}-\frac{1}{2}(\partial_{xx}\sigma)\sigma-\frac{\partial_t\sigma}{\sigma}\right)(s,X_s)\,ds\right\},
  \end{align}
  for $r\le{t}$\, and $D_rX_t=0$\, for $r>t$.
\end{theorem*}

\begin{proof}
  We omit the proof. For further details, refer to [Theorem 2.1 \cite{DGR}].
  \end{proof}

\section{Mean-Reverting CEV Model}

Following the construction in \cite{AE}, we will now prove that the mean-reverting {\CEV} model is Malliavin differentiable. The mean-reverting CEV model follows the {\SDE}
\begin{align}
d{\nu}_t = \kappa(\mu-{\nu}_t)\,dt+ {\theta}{{\nu}_t}^{\gamma}\,dW_t,\quad\gamma\in\left(\frac{1}{2}, 1\right),\label{31}
\end{align}
with ${\nu}_0 = \nu>0$ and where $\mu$, $\kappa$ and $\theta\,>0$. In \cite{AE}, Alos and Ewald proved the Malliavin differentiability of the case $\gamma=\frac{1}{2}$ of \eqref{31}. In the case, the function $x^{\frac{1}{2}}$ is neither continuously differentiable in 0 nor Lipschitz continuous so they circumvented various problems by some transforming and approximating. However, in the case $\gamma\in\left(\frac{1}{2},1\right)$, there are more complex problems. Following \cite{AE}, we will extend their results.

\subsection{Existence and Uniqueness}
We will now prove that the solution to $\eqref{31}$ not only exists uniquely but is also positive a.s.
\begin{lemma*}\label{lem3.1}
There exists a unique strong solution to \eqref{31} which satisfies $P(\,\nu_t\ge0,\,\,{t\ge0}\,)=1.$ Moreover, let $\tau=\inf\{\,t\ge0;\,\,\nu_t=0\,\,\text{or}\,\,=\infty\}$ with $\inf\{\emptyset\}=\infty$. Then we have $P(\,\tau=\infty\,)=1$.
\end{lemma*}

\begin{proof}
Instead of \eqref{31}, consider the following
\begin{align}
dv_t = \kappa(\mu-v_t)\,dt+ {\theta}|v_t|^{\gamma}\,dW_t,\quad\gamma\in\left(\frac{1}{2}, 1\right).\label{32}
\end{align}
If we have concluded that the unique strong solution of \eqref{32} is positive a.s., then \eqref{32} coincides with \eqref{31}. The existence of non-explosive weak solution for \eqref{32} follows from the continuity and the sub-linear growth condition of drift and diffusion coefficients. Moreover, from \cite[Proposition 5.3.20, Corollary 5.3.23]{KS}, we have the pathwise uniqueness. From \cite[Proposition 5.2.13]{KS}, we can verify that the pathwise uniqueness holds for \eqref{32}.

We will now prove that the second claim is true. Let $\tau_v=\inf\{\,t\ge0;\,\,v_t=0\,\,\text{or}\,\,=\infty\}$ with $\inf\{\emptyset\}=\infty$. In order to use \cite[Theorem 5.5.29]{KS}, we verify that for a fixed number $c\in{\bf{R}}$,\,\,$\lim_{x\to0}p(x)=-\infty$ \,\,where $p(x)$ is defined as
$p(x)=\int_c^x\exp\left\{-2\int_c^y \frac{\kappa(\mu-z)}{\theta^2{z}^{2\gamma}}\,dz\right\}dy$. Since we have known that the solution $v_t$ of \eqref{32} does not explode at $\infty$, if we could prove that the above formula holds, we can claim that $P(\tau_v=\infty)=1$, that is, $P(\tau=\infty)=1$. We can assume without restriction that $x<1$ and let $c=1$. Then we have
\begin{align}
-2\int_1^y\frac{\kappa(\mu-z)}{\theta^2{z}^{2\gamma}}dz&=-2\int_1^y\frac{\kappa\mu}{\theta^2{z}^{2\gamma}}-\frac{\kappa}{\theta^2{z}^{2\gamma-1}}dz\notag\\
&=\frac{2\kappa\mu}{\theta^2(2\gamma-1)}\left[\frac{1}{z^{2\gamma-1}}\right]_y^1-\frac{2\kappa}{\theta^2(2-2\gamma)}\left[\frac{1}{z^{2\gamma-2}}\right]_y^1\notag\\
&\ge\frac{2\kappa\mu}{\theta^2(2\gamma-1)}\left(\frac{1}{y^{2\gamma-1}}-1\right)+\frac{2\kappa}{\theta^2(2-2\gamma)}.
\end{align}
Letting $w=y^{-1}$, we can calculate $p(x)$. From the last inequality, there exists a constant $C>0$ satisfying the following inequality and then we have as \,\,$x\to0$,
\begin{align}
p(x)&\le-C\int_x^1\exp\left\{\frac{2\kappa\mu}{\theta^2(2\gamma-1)}\left(\frac{1}{y^{2\gamma-1}}\right)\right\}dy\notag\\
&=-C\int_1^{\frac{1}{x}}\frac{1}{w^2}\exp\left\{\frac{2\kappa\mu}{\theta^2(2\gamma-1)}w^{2\gamma-1}\right\}dw\to-\infty.
\end{align}
\end{proof}

\subsection{$L^p$-Integrability}
Consider the {\SDE}
\begin{align}
d{\nu}_t&=b({\nu}_t)\,dt + {\theta}{\nu}_t^{\gamma}\,dW_t,\label{33}
\end{align}
with ${\nu}_0=x>0$, where $b$ is such that $b(0)>0$ and satisfies the Lipschitz condition, ${\theta}>0$ and ${\gamma}\in\left(\frac{1}{2}, 1\right)$. The following lemma ensures the existence of its moments of any order.
\begin{lemma*}\label{lem3.2}
Consider the solution of the \eqref{33}. For any $p\geq0$, we have $E\big[\,\sup_{t\in[0, T]}{\nu}_t^{p} \,\big]<\infty$ and $E\big[\,\sup_{t\in[0, T]}{\nu}_t^{-p} \,\big]<\infty$.
\end{lemma*}

\begin{proof}
At first we consider the positive moments. We define the stopping time ${\tau}_n=\displaystyle\inf{\{\,0\le{t}\le{T};\,\,{\nu}_t\ge{n}\,\}}$ with $\displaystyle\inf\{\,\emptyset\,\}=\infty$. By It\^{o}'s formula,
\begin{align}
{\nu}_{t\wedge{\tau}_n}^{p}&={x}^{p}+\int_{0}^{t\wedge{\tau}_n}p{\nu}_s^{p-1}\,d{\nu}_s+\frac{1}{2}\int_0^{t\wedge{\tau}_n}p(p-1){\nu}_s^{p-2}\,(d{\nu}_s)^2\notag\\
&\le{x}^{p}+p\int_{0}^{t\wedge{\tau}_n}{\nu}_s^{p-1}b({\nu}_s)\,ds+p{\theta}\int_{0}^{t\wedge{\tau}_n}{\nu}_s^{{p-1}+{\gamma}}\,dW_s+\frac{p(p-1){\theta}^2}{2}\int_0^{t\wedge{\tau}_n}{\nu}_s^{{p-2}+2{\gamma}}\,ds.
\end{align}
From the Lipschitz condition of the drift function $b(x)$, there exists a positive constant $K$ which satisfies $b({\nu}_s)\le K{\nu}_s + b(0)$.
By the above inequality and Young's inequality, we have
\begin{align}
E[\,{\nu}_{t\wedge{\tau}_n}^{p}\,]\le&{x}^{p}+pE\left[\,\int_{0}^{t\wedge{\tau}_n}{\nu}_s^{p-1}b({\nu}_s)\,ds\,\right]+\frac{p(p-1){\theta}^2}{2}E\left[\,\int_0^{t\wedge{\tau}_n}{\nu}_s^{{p-2}+2{\gamma}}\,ds\,\right]\notag\\
\le&{x}^{p}+pKE\left[\,\int_{0}^{t\wedge{\tau}_n}{\nu}_s^p\,ds\,\right]+pE\left[\,\int_{0}^{t\wedge{\tau}_n}{\nu}_s^{p-1}b(0)\,ds\,\right]\notag\\
&+\frac{p(p-1){\theta}^2}{2}E\left[\,\int_0^{t\wedge{\tau}_n}{\nu}_s^{{p-2}+2{\gamma}}\,ds\,\right]\notag\\
\le&{x}^{p}+pKE\left[\,\int_{0}^{t\wedge{\tau}_n}{\nu}_s^p\,ds\,\right]+p\frac{E\left[\int_{0}^{t\wedge{\tau}_n}{\nu}_s^p\,ds\right]}{\frac{p}{p-1}}\notag\\
&+p\frac{b(0)^p}{p}+\frac{p(p-1){\theta}^2}{2}\left(\frac{E\left[\int_0^{t\wedge{\tau}_n}{\nu}_s^{p}\,ds\right]}{\frac{p}{p-2+2{\gamma}}}+\frac{1}{\frac{p}{2-2{\gamma}}}\right)\notag\\
=&C+{C'}\int_0^tE[\,{\nu}_{s\wedge{\tau}_n}^{p}]\,ds.
\end{align}
By Gronwall's lemma, we can have $E[\,{\nu}_{t\wedge{\tau}_n}^{p}]\le{C}\exp\{C't\}$, where both $C$ and $C'$ do not depend on $n$. As $n\to\infty$, we can obtain the result.
Next we consider the negative moments. Define the stopping time as ${\tau}_n=\displaystyle\inf{\left\{0\le{t}\le{T};\,\,{\nu}_t\le\frac{1}{n}\right\}}$, with $\displaystyle\inf\{\emptyset\}=\infty$. By It\^{o}'s formula, we have
\begin{align}
{\nu}_{t\wedge{\tau}_n}^{-p}&={x}^{-p}+\int_{0}^{t\wedge{\tau}_n}(-p){\nu}_s^{-(p+1)}d{\nu}_s+\frac{1}{2}\int_0^{t\wedge{\tau}_n}p(p+1){\nu}^{-(p+2)}(d{\nu}_s)^2\notag\\
&={x}^{-p}-p\int_{0}^{t\wedge{\tau}_n}\frac{b({\nu}_s)}{{\nu}_s^{p+1}}\,ds-p{\theta}\int_{0}^{t\wedge{\tau}_n}\frac{1}{{\nu}_s^{(p+1)-{\gamma}}}\,dW_s+\frac{p(p-1){\theta}^2}{2}\int_0^{t\wedge{\tau}_n}\frac{1}{{\nu}_s^{2(1-{\gamma})+p}}\,ds.
\end{align}
Taking the expectation and using the Fubini's theorem, we have
\begin{align}
E[\,{\nu}_{t\wedge{\tau}_n}^{-p}]&={x}^{-p}-pE\left[\int_{0}^{t\wedge{\tau}_n}\frac{b({\nu}_s)\,ds}{{\nu}_s^{(p+1)}}\right]+\frac{{\theta}^2}{2}p(p+1)E\left[\int_0^{t\wedge{\tau}_n}\frac{\,ds}{{\nu}_s^{2(1-{\gamma})+p}}\right]\notag\\
&\le{x}^{-p}+pK\int_{0}^{t}E\left[\frac{1}{{\nu}_{s\wedge{\tau}_n}^p} ds\right]+E\left[\int_0^{t\wedge{\tau}_n}\left(\frac{p(p+1){\theta}^2}{2{\nu}_s^{2(1-{\gamma})+p}}-\frac{pb(0)}{{\nu}_s^{p+1}}\right)\,ds\right].
\end{align}
Here let $q(x)=\frac{p(p+1){\theta}^2}{2x^{2(1-{\gamma})+p}}-\frac{pb(0)}{x^{p+1}}$, then we can easily evaluate the boundedness for any $x>0$
\begin{align}
q(x)\le {D}\coloneqq\frac{p(2{\gamma}-1){\theta}^2}{2}\left\{(2(1-{\gamma})+p)\frac{{\theta}^2}{2b(0)}\right\}^{\frac{2(1-{\gamma})+p}{2{\gamma}-1}}.
\end{align}
Summarizing the calculation, we have $E[\,{\nu}_{t\wedge{\tau}_n}^{-p}] \le x^{-p}+Dt+pK\int_0^tE[\,{\nu}_{s\wedge{\tau}_n}^{-p}]\,ds$, and from Gronwall's lemma we finally have $E[\,{\nu}_{t\wedge{\tau}_n}^{-p}]\le(x^{-p}+Dt)\exp\{pKt\}$. Taking the limit $n \to \infty$, then $\displaystyle\lim_{n\to\infty}{\tau}_n=\infty\,\,\mbox{a.s.}$ so we have
$E[\,{\nu}_{t}^{-p}]\le(x^{-p}+Dt)\exp\{pKt\}$. Hence we can deduce the result.
\end{proof}

\begin{remark*}
Since the CEV model satisfies the assumptions of {\rm{Lemma 3.2}}, so the result holds for the CEV model.
\end{remark*}

\subsection{Transformation and Approximation}

We consider the process transformed as ${\sigma}_t\coloneqq{\nu}_t^{1-\gamma}$. By It\^o's formula, we have
\begin{align}
d{\sigma}_t=(1-\gamma)\left(\kappa\mu{\sigma}_t^{-\frac{\gamma}{1-\gamma}}-\frac{\gamma{\theta}^2}{2}\frac{1}{{\sigma}_t}-\kappa{\sigma}_t\right)\,dt+(1-\gamma){\theta}\,dW_t,\label{34}
\end{align}
with $\sigma_0=\nu^{1-\gamma}>0$. If $\sigma_t$ is the solution of the \SDE \eqref{34}, then we can prove that $\sigma_t^{\frac{1}{1-\gamma}}=\nu_t$ is also the solution of the \SDE \eqref{31} satisfying the initial condition $\sigma_0^{\frac{1}{1-\gamma}}=\nu_0$. By this transformation, we can replace \eqref{31} by \eqref{34} with the constant volatility term. In order to use Theorem 2.5, we must approximate ${\frac{1}{x}}$ and $x^{-\frac{\gamma}{1-\gamma}}$ by the Lipschitz continuous functions, respectively. For all $\epsilon>0$, define the continuously differentiable functions $\Phi$ and $\Psi$ as
\begin{align}
\Phi(x)&=\begin{cases}
x^{-\frac{\gamma}{1-\gamma}} &(\,\text{for $x \ge\epsilon$}\,),\\
-\frac{\gamma}{1-\gamma}{\epsilon}^{-\frac{1}{1-\gamma}}x+\frac{1}{1-\gamma}{{\epsilon}^{-\frac{\gamma}{1-\gamma}}} &(\,\text{for $x < \epsilon$}\,),
\end{cases}\\
\Psi(x)&=\begin{cases}
\frac{1}{x} &(\,\text{for $x \ge \epsilon$}\,),\\
-\frac{1}{{\epsilon}^2}x+\frac{2}{\epsilon} &(\,\text{for $x < \epsilon$}\,).
\end{cases}
\end{align}
For the functions $\Phi$ and $\Psi$, we can easily verify that for all $x\in{\bf{R}}$, $|{\Phi}'(x)|\le{\frac{\gamma}{1-\gamma}}{{\epsilon}^{-\frac{1}{1-\gamma}}}$ and $|{\Psi}'(x)|\le{\frac{1}{{\epsilon}^2}}$ and then we have that for all $x,y \in \bf{R}$, $|\Phi(x)-\Phi(y)|\le{\frac{\gamma}{1-\gamma}}{{\epsilon}^{-\frac{1}{1-\gamma}}}|x-y|$ and $|\Psi(x)-\Psi(y)|\le{\frac{1}{{\epsilon}^2}}|x-y|$. Moreover, note that for all $x\in{\bf{R}}_{+}$, $\Phi(x)\le{x}^{-\frac{\gamma}{1-\gamma}}$ and $\Psi(x)\le{\frac{1}{x}}$. Define our approximations $\sigma_t^{\epsilon}$ as the stochastic process following the {\SDE}
\begin{align}
d{\sigma_t^{\epsilon}}=(1-\gamma)\left(\kappa\mu{\Phi(\sigma_t^{\epsilon})}-\frac{\gamma{\theta}^2}{2}{\Psi(\sigma_t^{\epsilon})}-\kappa{\sigma}_t\right)\,dt+(1-\gamma){\theta}\,dW_t,\label{35}
\end{align}
with ${{\sigma}_0^{\epsilon}}={{\sigma}_0}$ for all $\epsilon>0$. The coefficients of the equation \eqref{35} are Lipschitz continuous because we can have for all $x,y\in {\bf{R}}$,
\begin{align}
&\left|\left(\kappa\mu\Phi(x)-\frac{\gamma{\theta}^2}{2}{\Psi(x)}-\kappa{x}\right)-\left(\kappa\mu\Phi(y)-\frac{\gamma{\theta}^2}{2}{\Psi(y)}-\kappa{y}\right)\right|\notag\\
\le&\kappa\mu|\Phi(x)-\Phi(y)|+\frac{\gamma{\theta}^2}{2}|\Psi(x)-\Psi(y)|+\kappa|x-y|\notag\\
\le& \left({\frac{\kappa\mu\gamma}{1-\gamma}}{{\epsilon}^{-\frac{1}{1-\gamma}}}+  \frac{\gamma{\theta}^2}{2}{\frac{1}{{\epsilon}^2}}+\kappa\right)|x-y|.
\end{align}
We will prove that ${\sigma}_t^{\epsilon}$ converges to $\sigma_t$ in $L^2(\Omega)$. First we prove that ${\sigma}_t^{\epsilon}$ converges to $\sigma_t$ pointwise.
\begin{lemma*}
The sequence ${\sigma}_t^{\epsilon}$ converges to ${\sigma}_t$ {\rm{a.s.}}, for all $t\in[0,T]$.
\end{lemma*}
\begin{proof}
  Define for all ${\epsilon}>0$ the stopping time as ${\tau}^{\epsilon}\displaystyle\coloneqq\inf\{\,0\le t \le T; \,{\sigma}_t\le{\epsilon}\,\}$ with $\{\,\emptyset\,\}=\infty$. By the definition of $\Phi$, $\Psi$, and ${\tau}^{\epsilon}$, we have
  \begin{align}
  |{\sigma_{t\wedge{\tau}^{\epsilon}}}&-{\sigma^{\epsilon}_{t\wedge{\tau}^{\epsilon}}}|\notag\\
  \le&(1-\gamma)\left(\kappa\mu\int_0^{{t\wedge{\tau}^{\epsilon}}}\left|{\sigma}_s^{-\frac{\gamma}{1-\gamma}}-{\Phi(\sigma_s^{\epsilon})}\right|\,ds+\frac{{\gamma}{\theta}^2}{2}\int_0^{{t\wedge{\tau}^{\epsilon}}}\left|\frac{1}{{\sigma}_s}-{\Psi(\sigma_s^{\epsilon})}\right|\,ds+\kappa\int_0^{t\wedge{\tau}^{\epsilon}}|{\sigma_s}-{\sigma^{\epsilon}_s}|\,ds\right)\notag\\
  \le&(1-\gamma)\left(\kappa\mu\frac{\gamma}{1-\gamma}{\epsilon}^{-\frac{\gamma}{1-\gamma}}+\frac{{\gamma}{\theta}^2}{2{\epsilon}^2}+\kappa\right)\int_0^t|{\sigma_{s\wedge{\tau}^{\epsilon}}}-{\sigma^{\epsilon}_{s\wedge{\tau}^{\epsilon}}}|\,ds.
  \end{align}
By Gronwall's lemma, ${\sigma_t}={\sigma^{\epsilon}_t}$ for $t < \tau^{\epsilon}$ and by Lemma 3.1 and the fact that ${\tau}^{{\epsilon}_1} \le {\tau}^{{\epsilon}_2}$ for ${{\epsilon}_1}\ge{{\epsilon}_2}$, we have $\displaystyle\lim_{\epsilon\to0}{\tau}^{\epsilon}=\infty$ a.s. so $\displaystyle\lim_{\epsilon\to0}{\sigma^{\epsilon}_t}={\sigma_t}$  for all $t\in[0,T]$.
  \end{proof}
Next we prove that there exist square integrable processes $u_t$ and $w_t$ with $u_t\le{\sigma_t^{\epsilon}}\le{w}_t$ for all $t\in[0,T]$. Actually, we will see that $w_t$ is $\sigma_t$. Before starting with the proof, we prove the following inequality.
\begin{lemma*}
For $\gamma\in\left(\frac{1}{2}, 1\right)$ and $a,b>0$, let $f(x)=ax^{-\frac{\gamma}{1-\gamma}}-\frac{b}{x}$. We have, for $x \in {\bf{R}}_{+}$,
\begin{align}
f(x)\ge{-}\left({\frac{a\gamma}{b(1-\gamma)}}\right)^{-\frac{\gamma}{2\gamma-1}}\frac{1-\gamma}{a(2\gamma-1)}.
\end{align}
\end{lemma*}
\begin{proof}
  By differentiating $f(x)$, we can easily have the result.
  \end{proof}
Consider $a=\kappa\mu$ and $b=\frac{\gamma{\theta}^2}{2}$ in the above inequality, then we can have the below result.
\begin{lemma*}\label{lem35}
  Let $u_t$ be the solution of the following {\SDE}
  \begin{align}
  du_t=(1-\gamma)(C-\kappa{u_t})\,dt+{\theta}(1-\gamma)\,dW_t,
    \end{align}
    with $u_0={\sigma}_0$, where $C=-\left({\frac{2\kappa\mu}{{\theta^2}(1-\gamma)}}\right)^{-\frac{\gamma}{2\gamma-1}}\frac{1-\gamma}{\kappa\mu(2\gamma-1)}$.
    Then $u_t\le{\sigma_t^{\epsilon}}\le{\sigma}_t$ {\rm{a.s.}} for all $t\in[0,T]$.
  \end{lemma*}

  \begin{proof}
    From the definitions of $\Phi$ and $\Psi$, $\kappa\mu\Phi(x)-\frac{\gamma{\theta}^2}{2}\Psi(x)\ge{C}$ for all ${x}\in{\bf{R}}_{+}$, that is, the drift coefficient of ${u_t}$ is smaller than one of ${\sigma}_t^{\epsilon}$. By Yamada-Watanabe's comparison lemma (see \cite[Proposition 5.2.18]{KS}) and Lemma 3.1, we have ${u_t}\le{\sigma}_t^{\epsilon}$ a.s..\\
    We prove the second inequality. In order to use Yamada-Watanabe's comparison lemma, we must prove that, for $x\in{\bf{R}}_{+}$,\,\,$\kappa\mu\Phi(x)-\frac{\gamma{\theta}^2}{2}\Psi(x)\le\kappa\mu{x}^{-\frac{\gamma}{1-\gamma}}-\frac{\gamma{\theta}^2}{2x}$. Let $g(x)\displaystyle\coloneqq\kappa\mu{x}^{-\frac{\gamma}{1-\gamma}}-\frac{\gamma{\theta}^2}{2x}-\kappa\mu\Phi(x)+\frac{\gamma{\theta}^2}{2}\Psi(x)$. We can easily verify \,$g(x)=\kappa\mu{x}^{-\frac{\gamma}{1-\gamma}}-\frac{\gamma{\theta}^2}{2x}+\left(\frac{\kappa\mu\gamma}{1-\gamma}{\epsilon}^{-\frac{1}{1-\gamma}}-\frac{\gamma{\theta}^2}{2{\epsilon}^2}\right)x-\frac{\kappa\mu}{1-\gamma}{\epsilon}^{-\frac{\gamma}{1-\gamma}}+\frac{\gamma{\theta}^2}{\epsilon}$, for $x<\epsilon$ and $g(x)=0$ for $x\ge\epsilon$. For all $x<\epsilon$, we have
     \begin{align}
     g'(x)&=-\frac{\kappa\mu\gamma}{1-\gamma}{x}^{-\frac{1}{1-\gamma}}+\frac{\gamma{\theta}^2}{2x^2}+\frac{\kappa\mu\gamma}{1-\gamma}{\epsilon}^{-\frac{1}{1-\gamma}}-\frac{\gamma{\theta}^2}{2{\epsilon}^2},\\
      g''(x)&=x^{-\frac{2-\gamma}{1-\gamma}}\left(\frac{\kappa\mu\gamma}{(1-\gamma)^2}-{\gamma{\theta}^2}x^{\frac{2\gamma-1}{1-\gamma}}\right).
    \end{align}
    Then there is a constant $\eta>0$ with $g''(\eta)<0$ for all $x<\eta$ and $g''(\eta)>0$ for all $x>\eta$. For $\epsilon<\eta$, $g'(x)$ is decreasing for all $x<\epsilon$. Then $g(\epsilon)=0$ and $g'(\epsilon)=0$ imply for all $x<{\epsilon}$, $g(x)>0$, that is, for $x\in{\bf{R}}_{+}$
    \begin{align}
    \kappa\mu\Phi(x)-\frac{\gamma{\theta}^2}{2}\Psi(x)\le\kappa\mu{x}^{-\frac{\gamma}{1-\gamma}}-\frac{\gamma{\theta}^2}{2x}.
    \end{align}
    By Yamada-Watanabe's comparison lemma, we have ${\sigma}_t^{\epsilon}\le{\sigma}_t$ a.s.
    \end{proof}
\begin{theorem*}
  For all $t\in[0,T]$, the sequence ${\sigma}_t^{\epsilon}$ converges to ${\sigma}_t$ in $L^2(\Omega)$.
  \end{theorem*}
\begin{proof}
From Lemma \ref{lem35}, we have \,\,$|{\sigma}_t^{\epsilon}-{\sigma}_t|\le|u_t-{\sigma}_t|\le|u_t|+|{\sigma}_t|$. Lemma \ref{lem3.2} implies $|\sigma_t|\in{L}^2(\Omega)$. Moreover, the Ornstein-Uhlenbeck process $u_t\in{L}^2(\Omega)$. By the dominated convergence theorem we can have the convergence.
  \end{proof}

\subsection{Malliavin Differentiability}
We will prove the Malliavin differentiability of both ${\sigma}_t$ and ${\nu}_t$. To do this, we consider our approximation sequence $\sigma^{\epsilon}_t$. The approximating {\SDE} \eqref{35} of $\sigma^{\epsilon}_t$ satisfies the assumption of Theorem 2.5, so we can prove the Malliavin differentiability of ${\sigma}^{\epsilon}_t$.
\begin{lemma*}
${\sigma}_t^{\epsilon}$ belongs to ${\bf{D}}^{1,2}$ and we have
  \begin{align}
  D_r{\sigma}_t^{\epsilon}=(1-\gamma)\theta\exp\left\{(1-\gamma)\int_r^t \left(\kappa\mu\Phi'(\sigma_s^{\epsilon})-\frac{\gamma{\theta}^2}{2}\Psi'(\sigma_s^{\epsilon})-\kappa \right)\,ds\right\},
  \end{align}
  for \,${r} \le {t}$, and $D_r{\sigma}_t^{\epsilon}=0$ \,for\, $r>{t}$.
\end{lemma*}
\begin{proof}
By Theorem 2.5, we have the result.
\end{proof}
We will now prove the Malliavin differentiability of $\sigma_t$. To start with, we prove some useful lemmas.
\begin{lemma*}\label{lem38}
For $\gamma\in(\frac{1}{2}, 1)$ and $a,b>0$, let $f(x)=ax^{-\frac{\gamma}{1-\gamma}}-\frac{b}{x}$, then for $x \in {\bf{R}}_{+}$ we have
\begin{align}
f'(x)\le\left(\frac{a\gamma}{2b(1-\gamma)^2}\right)^{-\frac{1}{2\gamma-1}}\left(\frac{a\gamma(2\gamma-1)}{2(1-\gamma)^2}\right).
\end{align}
  \end{lemma*}
\begin{proof}
  By differentiating $f'(x)$ we can easily have the result.
  \end{proof}

  By Lemma 3.7, considering the case where $a=\kappa\mu$ and $b=\frac{\gamma\theta^2}{2}$, we have for $x\ge\epsilon$
  \begin{align}
  \kappa\mu\Phi'(x)-\frac{\gamma{\theta}^2}{2}\Psi'(x)\le\left(\frac{\kappa\mu}{{\theta^2}(1-\gamma)^2}\right)^{-\frac{1}{2\gamma-1}}\left(\frac{\kappa\mu\gamma(2\gamma-1)}{2(1-\gamma)^2}\right)\coloneqq{\xi}.
  \end{align}
  We have for $x<\epsilon$,
  \begin{align}
  \kappa\mu\Phi'(x)-\frac{\gamma{\theta}^2}{2}\Psi'(x)={\epsilon}^{-\frac{1}{1-\gamma}}\left(-\frac{\gamma{\theta}^2}{2}{\epsilon}^{-\frac{2\gamma-1}{1-\gamma}}+\frac{\kappa\mu\gamma}{1-\gamma}\right),
  \end{align}
  so there exists a constant ${\epsilon}_0 >0$ such that for all ${\epsilon}<{\epsilon}_0$, $\kappa\mu\Phi'(x)-\frac{\gamma{\theta}^2}{2}\Psi'(x)<0$. Hence, for ${\epsilon}<{\epsilon}_0$, we have $\kappa\mu\Phi'(x)-\frac{\gamma{\theta}^2}{2}\Psi'(x)\le{\xi}$, for all $x\in{\bf{R}}_{+}$. Note that $\xi$ is independent of ${\epsilon}$. By this inequality, we have the following result.
\begin{lemma*}
We have for all $t\in[0,T]$ and $\epsilon<{\epsilon}_0$,
\begin{align}
|D_r{\sigma}_t^{\epsilon}|\le(1-\gamma)\theta\exp\{(1-\gamma)(\xi-\kappa)(t-r)\}.
\end{align}
\end{lemma*}
\begin{proof}
  When $r>t$,  $D_r{\sigma}_t^{\epsilon}=0$ so the result follows. Moreover when\, $r \le t$, putting above results together, we obtain the result.
  \end{proof}
Putting the scenarios together, we can prove the following.
\begin{theorem*}
$\sigma_t$ belongs to ${\bf{D}}^{1,2}$ and we have
\begin{align}
D_r{\sigma}_t=(1-\gamma)\theta\exp\left\{(1-\gamma)\int_r^t \left(-\frac{\kappa\mu\gamma}{1-\gamma}{\sigma_s^{-\frac{1}{1-\gamma}}}+\frac{\gamma{\theta}^2}{2\sigma_s^2}-\kappa\right)\,ds\right\},
\end{align}
for\, $r \le t$, and $D_r{\sigma}_t=0$\, for\, $r>{t}$.
\end{theorem*}
\begin{proof}
  We have proved that ${\sigma}_t^{\epsilon}\to{\sigma}_t$ in $L^2(\Omega)$ and ${\sigma}_t^{\epsilon}\in{\bf{D}}^{1,2}$. Moreover, by Lemma 3.8, we have $\displaystyle\sup_{\epsilon}E\left[\,\|D{\sigma}_t^{\epsilon}\|^2\right]<\infty$. Here $\sigma_t^{\epsilon}$ converges to $\sigma_t$ also pointwise, we can conclude that $\displaystyle D_r{\sigma}_t^{\epsilon}$ converges to $G\displaystyle\coloneqq(1-\gamma)\theta\exp\left\{(1-\gamma)\int_r^t \left(-\frac{\kappa\mu\gamma}{1-\gamma}{\sigma_s^{-\frac{1}{1-\gamma}}}+\frac{\gamma{\theta}^2}{2\sigma_s^2}-\kappa\right)\,ds\right\}$. Using the bounded convergence theorem, we can have that $D_r{\sigma}_t^{\epsilon}$ converges to $G$ in ${L}^2(\Omega;H)$. Hence by Lemma 2.4, we can conclude that ${\sigma}_t\in{\bf{D}}^{1,2}$ and $D_r{\sigma}_t=G$.
  \end{proof}

Moreover we can prove the following Malliavin differentiability in more detail.

\begin{theorem*}
  For all $p\ge1$, \,$\sigma_t$ belongs to ${\bf{D}}^{1,p}$, that is, \,$\sigma_t$ belongs to ${\bf{D}}^{1,\infty}$.
  \end{theorem*}
\begin{proof}
We only have to prove that $\|{\sigma}_t\|^p_{1,p}<\infty$. We have
\begin{align}
\|{\sigma}_t\|^p_{1,p}=&E[\,{\sigma}_t^p]+E\left[\left|\int_0^T (D_r{\sigma}_t)^2\,dr\right|^{\frac{p}{2}}\right]\notag\\
=&E[\,{\nu}_t^{(1-\gamma)p}]\notag\\
&+E\left[\left|\int_0^T \left((1-\gamma)\theta\exp\left\{(1-\gamma)\int_r^t \left(-\frac{\kappa\mu\gamma}{1-\gamma}{\sigma_s^{-\frac{1}{1-\gamma}}}+\frac{\gamma{\theta}^2}{2\sigma_s^2}-\kappa\right)\,ds\right\}\right)^2\, dr\right|^{\frac{p}{2}}\right]\notag\\
\le&E[\,{\nu}_t^{(1-\gamma)p}]+E\left[\left|\int_0^T (1-\gamma)^2{\theta}^2\exp\{2(1-\gamma)(\xi-\kappa)(t-r)\}dr\right|^{\frac{p}{2}}\right]\notag\\
=&E[\,{\nu}_t^{(1-\gamma)p}]+\left(\frac{(1-\gamma){\theta}^2(1-\exp\{-2(1-\gamma)(\xi-\kappa)T\})}{2(\xi-\kappa)}\right)^{\frac{p}{2}}\exp\{p(1-\gamma)(C-\kappa)t\}.
\end{align}
Hence we can conclude that  $\|{\sigma}_t\|^p_{1,p}<\infty$.
  \end{proof}

By the chain rule, we can conclude that $\nu_t$ is also Malliavin differentiable.

\begin{theorem*}
  For all $p\ge1$, $\nu_t$ belongs to ${\bf{D}}^{1,\infty}$ and the Malliavin derivative is given by
  \begin{align}
  D_r{\nu}_t=\theta{\nu}_t^{\gamma}\exp\left\{(1-\gamma)\int_r^t \left(-\frac{\kappa\mu\gamma}{(1-\gamma)\nu_s}+\frac{\gamma{\theta}^2}{2{\nu_s}^{2(1-\gamma)}}-\kappa\right)\,ds\right\},
  \end{align}
  for $r \le t$, and $D_r\nu_t=0$ \,for \,$r>t$.
  \end{theorem*}
\begin{proof}
Consider only the case where $r \le t$. Similarly, we can easily prove the case where $r > t$. We have shown that $\nu_t=\sigma_t^{\frac{1}{1-\gamma}}$ and $\sigma_t\in{\bf{D}}^{1,\infty}$. By Lemma 2.5, we have
   \begin{align}
   D_r{\nu}_t=D_r{\sigma}^{\frac{1}{1-\gamma}}_t=\theta{\nu}_t^{\gamma}\exp\left\{(1-\gamma)\int_r^t \left(-\frac{\kappa\mu\gamma}{(1-\gamma)\nu_s}+\frac{\gamma{\theta}^2}{2{\nu_s}^{2(1-\gamma)}}-\kappa\right)\,ds\right\}.
   \end{align}
   For all $p\ge1$, using Young's inequality and the fact $\nu_t\in{L}^p(\Omega)$ and $\sigma_t\in{\bf{D}}^{1,\infty}$, we can prove that $\nu_t$ belongs to ${\bf{D}}^{1,\infty}$. Indeed, we have
  \begin{align}
  \|{\nu}_t\|^p_{1,p}=&E[\,{\nu}_t^p]+E\left[\left|\int_0^T (D_r{\nu}_t)^2dr\right|^{\frac{p}{2}}\right]=E[\,{\nu}_t^p]+\left(\frac{1}{1-\gamma}\right)^pE\left[{\sigma}_t^{\frac{p\gamma}{1-\gamma}}\left|\int_0^T(D_r{\sigma}_t)^2dr\right|^{\frac{p}{2}}\right]\notag\\
  \le&E[\,{\nu}_t^p]+\left(\frac{1}{1-\gamma}\right)^pE\left[\frac{1}{2}{\sigma}_t^{\frac{2p\gamma}{1-\gamma}}+\frac{1}{2}\left|\int_0^T(D_r{\sigma}_t)^2dr\right|^{p}\right]\notag\\
  =&E[\,{\nu}_t^p]+\left(\frac{1}{1-\gamma}\right)^pE\left[\frac{1}{2}{\nu}_t^{2p\gamma}+\frac{1}{2}\left|\int_0^T(D_r{\sigma}_t)^2dr\right|^{p}\right]\notag\\
  <&\infty.
  \end{align}
  \end{proof}

\section{CEV-Type Heston Model and Greeks}
We will now consider the CEV-type Heston model and Greeks. Fourni\'e et al.\,\,introduced new numerical methods for calculating Greeks using {\MC} for the first time in 1999 (see \cite{F1}). We call this methods Malliavin Monte-Carlo methods. They focused on models with Lipschitz continuous coefficients, and then a lot of researchers have considered Malliavin Monte-Carlo methods to compute Greeks. However, lately, there is need to focus on models with non-Lipschitz coefficients such as stochastic volatility models. In 2008, Alos and Ewald proved that the Cox-Ingersoll-Ross model was Malliavin differentiable (see \cite{AE}). We apply {\MC} for calculating Greeks of the CEV-type Heston model which is one of the important in business but mathematically complex models. Basically, we consider the European option but we can easily extend this result to other options.

\subsection{Greeks}
We introduce the concept of Greeks. For example, consider a European option with payoff function $\phi$ depending on the final value of the underlying asset $S_T$ where $S_t$ denotes a stochastic process expressing the asset and $T$ denotes the maturity of the option. The price $V$ is given by \,$V=E[\,e^{-rT}\phi(S_T)]$\, where $r$ is the risk-free rate. We can estimate this by Monte-Carlo simulations. Greeks are derivatives of the option price $V$ with respect to the parameters of the model. Greeks are the useful measure for the portfolio risk management by traders in financial institutions. Most of financial institutions estimate Greeks by finite difference methods. However, there are some demerits. For examples, the numerical results depend on the approximation parameters and, in the case where $\phi$ is not differentiable, this methods do not work well. In \cite{F1}, Founi\'e et al.\,\,gave the new methods to circumvent these problems. The idea is that we calculate Greeks by multiplying the weight, so-called Malliavin weight, as following
\begin{align}
\frac{\partial{V(x)}}{\partial\alpha}=E[\,e^{-rT}\phi(S_T){\cdot}({\rm{weight}})].
\end{align}
This methods are much useful since we do not require the differentiability of the payoff function $\phi$. Instead, there is need to assume that the underlying assert $S_t$ is Malliavin differentiable. From Theorem 2.2, we find that the solution of the {\SDE} with Lipschitz continuous coefficients are Malliavin differentiable. However, if a model under consideration becomes more complex just like the CEV-type Heston model, we could not apply this Malliavin methods. Through Section 4, we consider the Malliavin differentiability of the CEV-type Heston model in order to give formulas for Greeks, in particular, Delta and Rho. Here, Delta $\Delta$ and Rho $\varrho$ respectively measure the sensitivity of the option price with respect to the initial price and the risk-free rate. In particular, $\Delta$ is one of the most important Greeks which also describes the replicating portfolio.

\subsection{CEV-Type Heston Model}
In \cite{H}, Heston supposed that the stock price $S_t$ follows the {\SDE}
\begin{align}
dS_t=S_t (r\,dt+\sqrt{\nu_t}\,dB_t),
\end{align}
where $B_t$, $r$ and $\sqrt{\nu_t}$ respectively mean a {\BM}, the risk-free rate and the volatility. Moreover Heston assumed that the volatility process $\nu_t$ becomes a mean-reverting stochastic process of the form
\begin{align}
d{\nu}_t = \kappa(\mu-{\nu}_t)\,dt+ {\theta}{\sqrt{\nu}_t}\,dW_t,
\end{align}
where $W_t$, $\mu$, $\kappa$ and $\theta$ respetively mean a {\BM}, the long-run mean, the rate of mean reversion and the volatility of volatility. This model is called the Cox-Ingersoll-Ross model. Here $B_t$ and $W_t$ are two correlated {\BM}s with
\begin{align}
\,dB_t\,dW_t={\rho}\,dt,\quad\rho\in(-1,1),
\end{align}
where $\rho$ is the correlation coefficient between two {\BM}s. Moreover we assume that the dynamics following {\SDE}s (4.1), (4.2), and (4.3) are satisfied under the risk neutral measure. However even the Heston model can not grasp the fluctuation of the volatility accurately. In \cite{AP}, Andersen and Piterbarg extended the Heston model to the model of which dynamics follow
\begin{align}
dS_t=&S_t (r\,dt+\sqrt{\nu_t}\,dB_t),\\
d{\nu}_t =& \kappa(\mu-{\nu}_t)\,dt+ {\theta}{{\nu}_t}^{\gamma}\,dW_t,\quad{\gamma}\in\left[\frac{1}{2},1\right],\\
\,dB_t\,dW_t=&{\rho}\,dt,\quad \rho\in(-1,1),
\end{align}
with the initial conditions $S_0=x$ and $\nu_0=\nu$. We call this model the CEV-type Heston model. For the equation (4.5) with $\gamma=1$, the Malliavin differentiability obviously follows by Theorem 2.2. In the case $\gamma=\frac{1}{2}$, Alos and Ewald proved the Malliavin differentiability in \cite{AE}. In Section 3, we have proved the Malliavin differentiability in the case $\gamma\in(\frac{1}{2},1)$.
Fron now on, we concentrate on $\gamma\in\left(\frac{1}{2},1\right)$. In order to give the formulas for the CEV-type Heston model, we will now prove the Malliavin differentiability of the model. Before considering the Malliavin differentiability, we now prove that there is a following {\BM} $\hat{W}_t$ which will become useful later.

\begin{lemma*}
There exists a {\BM} $\hat{W}_t$ independent of $W_t$ with\,\,$B_t=\rho{W_t} +\sqrt{1-{\rho}^2}\hat{W}_t$.
\end{lemma*}

\begin{proof}
  From the definition of $\hat{W}_t$, we have \,$\hat{W}_t=\frac{1}{\sqrt{1-{\rho}^2}}B_t-\frac{\rho}{\sqrt{1-{\rho}^2}}{W_t}$. At first we prove that $\hat{W}_t$ is independent of $W_t$. Since we easily have\,\,$E[\,W_t\hat{W}_t]=0$, so $\hat{W}_t$ is independent of $W_t$. Using L\^{e}by's theorem, we conclude \, $\hat{W}_t$ is a {\BM}. We can easily verify that $\hat{W}_t$ is also martingale. Consider the quadratic variation ${\langle\hat{W}\rangle}_t$ of $\hat{W}_t$. Then we have
  \begin{align}
  {\langle\hat{W}\rangle}_t={\left\langle{\frac{1}{\sqrt{1-{\rho}^2}}}B-\frac{\rho}{\sqrt{1-{\rho}^2}}W\right\rangle}_t=t.
  \end{align}
  Hence by the L\^{e}vy's theorem, $\hat{W}_t$ is a {\BM}.
\end{proof}
Instead of the dynamics (4.5), (4.6) and (4.7), replacing $B_t$ by $\hat{W}_t$, then we can consider the following
\begin{align}
dS_t=&S_t (r\,dt+\sqrt{\nu_t} (\rho{\,dW_t} +\sqrt{1-{\rho}^2}d\hat{W}_t)),\\
d{\nu}_t =& \kappa(\mu-{\nu}_t)\,dt+ {\theta}{{\nu}_t}^{\gamma}\,dW_t,
\end{align}
where $W_t$ and $\hat{W}_t$ are independent. Note that we assume that $S_t$ and $\nu_t$ follow the dynamics (4.7) and (4.8) under the risk neutral measure.

\subsection{Arbitrage}
Under the real measure, the CEV-type Heston model follows the following dynamics
\begin{align}
dS_t=&S_t (u\,dt+\sqrt{\nu_t} (\rho{\,dW_t} +\sqrt{1-{\rho}^2}d\hat{W}_t)),\\
d{\nu}_t =& \kappa(\mu-{\nu}_t)\,dt+ {\theta}{{\nu}_t}^{\gamma}\,dW_t,
\end{align}
where $W_t$ and $\hat{W}_t$ are independent. Here $u$ denotes the expected return of $S_t$. In business, $u$ is assumed to equal to the risk free rate. In order to do this, we will change the real measure $P$ to the measure $Q$ called the risk-neutral measure. We consider the arbitrage but this problem is complicated, since the volatility is not tractable. However, we obtain the following theorem.
\begin{theorem*}
The CEV-type Heston model following {\rm{(4.9)}} and {\rm{(4.10)}} is free of arbitrage and there is a risk-neutral measure $Q$
\begin{align}
dS_t=&S_t (r\,dt+\sqrt{\nu_t} (\rho{\,dW_t} +\sqrt{1-{\rho}^2}d\hat{W}_t)),\\
d{\nu}_t =& \kappa(\mu-{\nu}_t)\,dt+ {\theta}{{\nu}_t}^{\gamma}\,dW_t.
\end{align}
\end{theorem*}

\begin{proof}
We consider the interval $[0,T]$. First we solve the equation $u-r=\sqrt{\nu_t}\sqrt{1-\rho^2}z^1_t+\sqrt{\nu_t}{\rho}z^2_t$.
In order to solve this, we put $z^2_t=0$. From Lemma 3.1, $\nu_t$ is positive a.s. so we have $z^1_t = \frac{u-r}{\sqrt{1-\rho^2}\sqrt{\nu_t}}$. Here $z^1_t$ is obviously progressively measurable. Moreover, we can easily see that $z^1_t$ is locally bounded and in $L^2(\Omega)$. Let $\epsilon(M)_t\coloneqq\exp\left\{M_t-\frac{1}{2}\langle M \rangle_t \right\}$ where $M_t=\displaystyle -\int_0^t z^1_s\,d\hat{W}_s$. It is well-known that if we can prove that $\epsilon(M)_t$ is a martingale, then the market is free of arbitrage and under the risk neutral measure $Q$ with $Q(A)=E^p[\,\epsilon(M)_T\cdot{\bf{1}}_A\,],{\quad}A\in{\cal{F}}_T$. Note that $\hat{W}_t$ is replaced by $\bar{W}_t$ which is a {\BM} under $Q$. Here we must prove that for all $t\ge0$, $E[\epsilon(M)_t]=1$. Fix $t\ge0$ and let $\tau_n=\inf\{s\ge0;\, |z^1_t| \ge n\}$ with $\inf\{\emptyset\}=\infty$. Here $z^1_{\tau_n \wedge \cdot}$ is bounded, so we have $\displaystyle\langle M^{\tau_n}\rangle_t=\int_0^t (z^1_{s \wedge \tau_n})^2\,ds$ is bounded. From Novikov's criteria, we have that $\epsilon(M^{\tau_n})_t$ is a uniformly integrable martingale for any $t\ge0$. Moreover, from the continuity of $z^1_t$ and Lemma \ref{lem3.1}, $\tau_n$ increases to infinity. Since $\epsilon(M)_t$ is positive a.s., $\epsilon(M)_t\cdot{\bf{1}}_{\{t \le \tau_n\}}$ converges to $\epsilon(M)_t$ as $n\to\infty$, and then by using the monotone convergence theorem
\begin{align}
E[\epsilon(M)_t]=\lim_{n\to\infty}E[\epsilon(M)_t\cdot{\bf{1}}_{\{t \le \tau_n\}}].
\end{align}
Here we have $\epsilon(M)_t\cdot{\bf{1}}_{\{t \le \tau_n\}}=\epsilon(M^{\tau_n \wedge t})_T\cdot{\bf{1}}_{\{t \le \tau_n\}}$, so letting $Q^n$ be the measure satisfying $\frac{dQ^n}{dP}=\epsilon(M^{t \wedge \tau_n})_T$, and then we have
 \begin{align}
 E[\epsilon(M)_t]=\lim_{n\to\infty}E[\epsilon(M)_t]\cdot{\bf{1}}_{\{t \le \tau_n\}}=\lim_{n\to\infty}E[\epsilon(M^{t \wedge \tau_n})_T] \cdot{\bf{1}}_{\{t \le \tau_n\}}=\lim_{n\to\infty}Q^n(t \le \tau_n).
 \end{align}
We must prove $\displaystyle\lim_{n\to\infty}Q^n(t \le \tau_n)=1$. First we prove $Q^n(t \le \tau_n)=P(t \le \tau_n)$. From Girsanov's theorem, the processes $\hat{W}_t+\displaystyle\int_0^t z^1_s \cdot{\bf{1}}_{[0,\tau_n \wedge t]}(s)\,ds$ and $W_t$ are ${\cal{F}}_t$-{\BM}s under the measure $Q^n$. Note that $W_t$ is an ${\cal{F}}_t$-adapted {\BM} under $Q^n$ for all $n$. We have known that under the measure $P$, $\nu_t$ follows the equation
\begin{align}
d{\nu}_t =& \kappa(\mu-{\nu}_t)\,dt+ {\theta}{{\nu}_t}^{\gamma}\,dW_t.
\end{align}
Integrals under $P$ and $Q^n$ are the same, so $\nu_t$ also satisfies the above {\SDE} under $Q^n$. From Lemma 3.1, the solution $\nu_t$ is unique. Hence the distribution of $\nu_t$ under the measure $Q^n$ must be the same as the distribution of $\nu_t$ under the measure $P$, and then we can conclude that the distribution $\tau_n$ is the same under $P$ and $Q^n$, that is, $Q^n(t \le \tau_n)=P(t \le \tau_n)$. Since $\tau_n$ tends to $\infty$ a.s., $\displaystyle\lim_{n\to\infty}P(t\le\tau_n)=1$. Hence we can conclude $E[\epsilon(M)_t]=1$ and $\epsilon(M)_t$ is a martingale. Then the market is free of arbitrage.
\end{proof}
This theorem implies that the dynamics for the volatility process is preserved, and the drift term of the underlying asset is changed from $u$ to $r$. In the sequel, we will consider the CEV-type Heston model under the risk-neutral measure denoted by $P$ not by $Q$.

\subsection{Malliavin Differentiability of the CEV-Type Heston Model (Logarithmic Price)}

From now on, we denote by $D$ and $\hat{D}$ two Malliavin derivatives {\WRT} $W_t$ and $\hat{W}_t$, respectively. We now consider the logarithmic price  $X_t\coloneqq\log{S_t}$. First, we will prove that $X_t$ is Malliavin differentiable. By It\^o's formula, we have
\begin{align}
  dX_t=\left(r-\frac{\nu_t}{2}\right)\,dt+\sqrt{\nu_t}
  \,dB_t,
  \end{align}
with $X_0=\log{x}$. Here $\sqrt{\nu_t}$ is neither differentiable at $\nu_t=0$ in 0 nor Lipschitz continuous. Hence we will now approximate this {\SDE} by one with Lipschitz continuous coefficients and prove the Malliavin differentiability of $X_t$. Let
\begin{align}
  \phi^{\epsilon}(x)=\left\{\begin{array}{ccc}
    e^x-e^{\epsilon}+\epsilon & (x<\epsilon),\\
    x & (\epsilon\le{x}<\frac{1}{\epsilon}),\\
      -e^{-x+\frac{1}{\epsilon}}+\frac{1}{\epsilon}+1 & (\frac{1}{\epsilon}\le{x}).
  \end{array}
  \right.
  \end{align}

  Here we can easily verify that $\phi^{\epsilon}(x)$ is bounded and continuously differentiable. Moreover we can verify that  both $(\phi^{\epsilon}(x))^{\frac{1}{1-\gamma}}$ and $(\phi^{\epsilon}(x))^{\frac{1}{2(1-\gamma)}}$ are Lipschitz continuous. In Section 3, we have used the stochastic process $\sigma^{\epsilon}$ with Lipschitz continuous coefficients, in stead of $\nu_t$. We will now prove the Malliavin differentiability of the two stochastic processes $\sigma^{\epsilon}$ and the following approximation process $X^{\epsilon}$ of $X$ with Lipschitz coefficients. Naturally, instead of $X_t$, we consider the following {\SDE}
\begin{align}
  dX^{\epsilon}_t=\left(r-\frac{1}{2}(\phi^{\epsilon}(\sigma^{\epsilon}_t))^{\frac{1}{1-\gamma}}\right)\,dt+(\phi^{\epsilon}(\sigma^{\epsilon}_t))^{\frac{1}{2(1-\gamma)}}
  \,dB_t,
\end{align}
with $X^{\epsilon}_0=\log{x}$.
\begin{lemma*}
  We have \,$X^{\epsilon}_t\to{X_t}\quad{\rm{in}}\,\,L^2(\Omega)$.
  \end{lemma*}

\begin{proof}
  From the inquality\,$(a+b)^2\le2(a^2+b^2)$, we have
  \begin{align}
  |X^{\epsilon}_t-X_t|^2 \le \frac{1}{2}\left|\int_0^t \left(\sigma_s^{\frac{1}{1-\gamma}}-(\phi^{\epsilon}(\sigma^{\epsilon}_s))^{\frac{1}{1-\gamma}}\right)\,ds\right|^2+2\left|\int_0^t \left(\sigma_s^{\frac{1}{2(1-\gamma)}}-(\phi^{\epsilon}(\sigma^{\epsilon}_s))^{\frac{1}{2(1-\gamma)}}\right)\,dB_s\right|^2.
\end{align}
We have using Cauchy-Schwarz's inequality and It\^o's isometry,
\begin{align}
    &E\left[|X^{\epsilon}_t-X_t|^2\right]\notag\\
    \le&\frac{t^2}{2} E\left[\int_0^t\left|\sigma_s^{\frac{1}{1-\gamma}}-(\phi^{\epsilon}(\sigma^{\epsilon}_s))^{\frac{1}{1-\gamma}}\right|^2\,ds\right]+2E\left[\int_0^t\left|\sigma_s^{\frac{1}{2(1-\gamma)}}-(\phi^{\epsilon}(\sigma^{\epsilon}_s))^{\frac{1}{2(1-\gamma)}}\right|^2\,ds\right].
  \end{align}
  For the second term, since both $\sigma_t$ and $\sigma^{\epsilon}_t$ are positive a.s. and for $a,b>0$, \,$(a-b)^2\le|a^2-b^2|$,
  we have
  \begin{align}
    E\left[|X^{\epsilon}_t-X_t|^2\right]\le\frac{t^2}{2} E\left[\int_0^t\left|\sigma_s^{\frac{1}{1-\gamma}}-(\phi^{\epsilon}(\sigma^{\epsilon}_s))^{\frac{1}{1-\gamma}}\right|^2\,ds\right]+2E\left[\int_0^t\left|\sigma_s^{\frac{1}{1-\gamma}}-(\phi^{\epsilon}(\sigma^{\epsilon}_s))^{\frac{1}{1-\gamma}}\right|\,ds\right].
  \end{align}
  By the scenarios in Subsection 3.3 and Subsection 3.4, we have that for almost all $\omega\in\Omega$ there exists a positive constant $\epsilon_0(\omega)$ such that for all $\epsilon<{\epsilon}_0(\omega)$, \,$\epsilon<\sigma^{\epsilon}_t(\omega)=\sigma_t(\omega)<\frac{1}{\epsilon}$. For such $\epsilon$, let
  \begin{align}
  \tau=\inf\left\{\,t\ge0;\,\,\sigma_t^{\epsilon}={\epsilon}\,\,{\rm{or}}\,\,\sigma_t^{\epsilon}=\frac{1}{\epsilon}\,\right\},
  \end{align}
  with $\inf\{\emptyset\}=\infty$, then we have $|\sigma^{\epsilon}_t-\sigma_t|=0$ for $t\le\tau$. Hence we can have $\left|(\phi^{\epsilon}(\sigma^{\epsilon}_t))^{\frac{1}{1-\gamma}}-\sigma_t^{\frac{1}{1-\gamma}}\right|=0 \quad \rm{a.s.}$,
   for $t\le\tau$. And then we can have $\tau\to\infty$ as $\epsilon\to0$,\,$\left|(\phi^{\epsilon}(\sigma^{\epsilon}_t))^{\frac{1}{1-\gamma}}-\sigma_t^{\frac{1}{1-\gamma}}\right|=0\quad \rm{a.s.}$ \,for all $t \ge 0 $. Since $\phi^{\epsilon}(x)\le{x}$ and $\sigma^{\epsilon}_t\le\sigma_t$, \,$\left|(\phi^{\epsilon}(\sigma^{\epsilon}_t))^{\frac{1}{1-\gamma}}-\sigma_t^{\frac{1}{1-\gamma}}\right|\le2\left|\sigma_t\right|^{\frac{1}{1-\gamma}}$. Here $\sigma_t$ is $L^p$-integrable for all $p\ge1$ so we can conclude that for all $p\ge1$,\,$\left|(\phi^{\epsilon}(\sigma^{\epsilon}_t))^{\frac{1}{1-\gamma}}-\sigma_t^{\frac{1}{1-\gamma}}\right|=0\quad{\rm{in}}\,\,L^p(\Omega)$. We have from Fubini's theorem, \,$|X^{\epsilon}_t-X_t|\to0 \quad {\rm{in}}\,\,L^2(\Omega)$.
  \end{proof}

The following theorem implies that $X_t$ is Malliavin differentiable.
\begin{theorem*}
  $X_t$ belongs to ${\bf{D}}^{1,2}$ and the Malliavin derivatives are given by
  \begin{align}
    D_uX_t=&-\frac{1}{2(1-\gamma)}\int_u^t \sigma_t^{\frac{\gamma}{1-\gamma}}D_s\sigma_s\,ds+\rho\sigma_u^{\frac{1}{2(1-\gamma)}}+\frac{1}{2(1-\gamma)}\int_u^t\sigma_s^{\frac{2\gamma-1}{2(1-\gamma)}}D_u\sigma_s\,dB_s,\\
    \hat{D}_uX_t=&\sqrt{1-\rho^2}{\sigma_u}^{\frac{1}{2(1-\gamma)}},
    \end{align}
    for $u \le t$, and $D_uX_t=\hat{D}_uX_t=0$ for $u > t$.
\end{theorem*}

\begin{proof}
Since the coefficients of {\SDE}s for $\sigma^{\epsilon}_t$ and $X_t^{\epsilon}$ are Lipschitz continuous, we can use Theorem 2.2. At first, we can conclude that $\sigma^{\epsilon}_t\in{\bf{D}}^{1,\infty}$ and the derivatives are given by
\begin{align}
  D_u\sigma^{\epsilon}_t=&(1-\gamma)\theta\exp\left\{(1-\gamma)\int_u^t (\kappa\mu\Phi'(\sigma_s^{\epsilon})-\frac{\gamma{\theta}^2}{2}\Psi'(\sigma_s^{\epsilon})-\kappa)\,ds\right\},\\
  \hat{D}_u\sigma^{\epsilon}_t=&0,
  \end{align}
  for $u\le{t}$ and $D_u\sigma^{\epsilon}_t=\hat{D}_u\sigma^{\epsilon}_t=0$ for $u>t$. \\
  Moreover we can also conclude that $X_t^{\epsilon}\in{\bf{D}}^{1,\infty}$ and the derivatives are given by the following
\begin{align}
  D_uX^{\epsilon}_t=&\int_u^t D_u\left(r-\frac{1}{2}(\phi^{\epsilon}(\sigma^{\epsilon}_s))^{\frac{1}{1-\gamma}}\right)\,ds+\rho(\phi^{\epsilon}(\sigma^{\epsilon}_u))^{\frac{1}{2(1-\gamma)}}+\int_u^tD_u(\phi^{\epsilon}(\sigma^{\epsilon}_s))^{\frac{1}{2(1-\gamma)}}dBs\notag\\
  =&-\frac{1}{2(1-\gamma)}\int_u^t (\phi^{\epsilon}(\sigma^{\epsilon}_s))^{\frac{\gamma}{1-\gamma}}D_u\sigma_s^{\epsilon}\,ds+\rho(\phi^{\epsilon}(\sigma^{\epsilon}_u))^{\frac{1}{2(1-\gamma)}}\notag\\
  &+\frac{1}{2(1-\gamma)}\int_u^t(\phi^{\epsilon}(\sigma^{\epsilon}_s))^{\frac{2\gamma-1}{2(1-\gamma)}}D_u\sigma^{\epsilon}_sdBs,\\
  \hat{D}_uX^{\epsilon}_t=&\sqrt{1-\rho^2}(\phi^{\epsilon}(\sigma^{\epsilon}_u))^{\frac{1}{2(1-\gamma)}},
  \end{align}
  for $u \le t$, and $D_uX^{\epsilon}_t=\hat{D}_uX^{\epsilon}_t=0$ for $u>t$.

  We only consider the case $u\le{t}$. First we consider the Malliavin derivative $\hat{D}_uX^{\epsilon}_t$. By Lemma 4.2 and the proof, we have $(\phi^{\epsilon}(\sigma^{\epsilon}_t))^{\frac{1}{2(1-\gamma)}}\to\sigma_t^{\frac{1}{2(1-\gamma)}}$ in ${L}^2(\Omega;H)$ and $X^{\epsilon}_t{\to}X_t$ in $L^2(\Omega)$. Moreover, $(\phi^{\epsilon}(x))^{\frac{1}{2(1-\gamma)}}$ is bounded, so we can use Lemma 2.4. Hence we can conclude \,$\hat{D}_uX_t=\sqrt{1-\rho^2}{\sigma_u}^{\frac{1}{2(1-\gamma)}}$. We consider the Malliavin derivative $D_uX_t$. For the first term, we need prove
\begin{align}
\int_u^t (\phi^{\epsilon}(\sigma^{\epsilon}_s))^{\frac{\gamma}{1-\gamma}}D_u\sigma_s^{\epsilon}\,ds\to\int_u^t \sigma_t^{\frac{\gamma}{1-\gamma}}D_u\sigma_s\,ds \quad{\rm{in}}\,\, {L}^2(\Omega;H).
\end{align}
Here we have that
\begin{align}
&\left|\int_u^t (\phi^{\epsilon}(\sigma^{\epsilon}_s))^{\frac{\gamma}{1-\gamma}}D_u\sigma_s^{\epsilon}\,ds-\int_u^t \sigma_t^{\frac{\gamma}{1-\gamma}}D_u\sigma_t\,ds\right|\notag\\
=&\left| \int_u^t \left((\phi^{\epsilon}(\sigma^{\epsilon}_s))^{\frac{\gamma}{1-\gamma}}D_u\sigma_s^{\epsilon}-\sigma_t^{\frac{\gamma}{1-\gamma}}D_u\sigma_s\right)\,ds\right|\notag\\
=&\left|\int_u^t \left((\phi^{\epsilon}(\sigma^{\epsilon}_s))^{\frac{\gamma}{1-\gamma}}D_u\sigma_s^{\epsilon}-\sigma_t^{\frac{\gamma}{1-\gamma}}D_u\sigma_s^{\epsilon}\right)+\left(\sigma_t^{\frac{1}{1-\gamma}}D_u\sigma_s^{\epsilon}-\sigma_s^{\frac{1}{1-\gamma}}D_u\sigma_t\right)\,ds\right|\notag\\
\le&\left|\int_u^t \left((\phi^{\epsilon}(\sigma^{\epsilon}_s))^{\frac{\gamma}{1-\gamma}}D_u\sigma_s^{\epsilon}-\sigma_t^{\frac{\gamma}{1-\gamma}}D_u\sigma_s^{\epsilon}\right)\,ds\right|+\left|\int_u^t\left(\sigma_t^{\frac{\gamma}{1-\gamma}}D_u\sigma_s^{\epsilon}-\sigma_t^{\frac{\gamma}{1-\gamma}}D_u\sigma_s\right)\,ds\right|\notag\\
\le&\left|\int_u^t \left((\phi^{\epsilon}(\sigma^{\epsilon}_s))^{\frac{\gamma}{1-\gamma}}-\sigma_t^{\frac{\gamma}{1-\gamma}}\right)\,ds\right|\sup_{s\in[u,t]}\left|{D_u\sigma_s^{\epsilon}}\right|+(t-u)\sup_{s\in[u,t]}\left|D_u\sigma_s^{\epsilon}-D_u\sigma_s\right|\sup_{s\in[u,t]}\left|\sigma_s^{\frac{\gamma}{1-\gamma}}\right|.
\end{align}
This converges to $0$ in $L^2(\Omega)$ by the proof of Lemma 4.2, Lemma 3.8, Theorem 3.2, and Lemma 3.1. Hence we can conclude $\int_u^t (\phi^{\epsilon}(\sigma^{\epsilon}_s))^{\frac{\gamma}{1-\gamma}}D_u\sigma_s^{\epsilon}\,ds\to\int_u^t \sigma_t^{\frac{1}{1-\gamma}}D_u\sigma_s\,ds\quad{\rm{in}}\,\, {L}^2(\Omega;H)$.
For the second term, as well as the case for $\hat{D}_uX_t$, we can prove that \,$\rho(\phi^{\epsilon}(\sigma^{\epsilon}_u))^{\frac{1}{2(1-\gamma)}}\to\rho\sigma_t^{\frac{\gamma}{2(1-\gamma)}}\quad{\rm{in}}\,\,{L}^2(\Omega;H)$.
For the third term, we will prove \,$\int_u^t(\phi^{\epsilon}(\sigma^{\epsilon}_s))^{\frac{2\gamma-1}{2(1-\gamma)}}D_u\sigma^{\epsilon}_s\,dB_s\to\int_u^t\sigma_s^{\frac{2\gamma-1}{2(1-\gamma)}}D_u\sigma^{\epsilon}_s\,dB_s\quad{\rm{in}}\,\,{L}^2(\Omega;H)$.
We have from It\^o's isometry,
  \begin{align}
    &\left|\int_u^t(\phi^{\epsilon}(\sigma^{\epsilon}_s))^{\frac{2\gamma-1}{2(1-\gamma)}}D_u\sigma^{\epsilon}_s\,dB_s-\int_u^t\sigma_s^{\frac{2\gamma-1}{2(1-\gamma)}}D_u\sigma_s\,dB_s\right|^2\notag\\
  =&\int_u^t\left|\left((\phi^{\epsilon}(\sigma^{\epsilon}_s))^{\frac{2\gamma-1}{2(1-\gamma)}}D_u\sigma^{\epsilon}_s-\sigma_s^{\frac{2\gamma-1}{2(1-\gamma)}}D_u\sigma^{\epsilon}_s\right)+\left(\sigma_s^{\frac{2\gamma-1}{2(1-\gamma)}}D_u\sigma^{\epsilon}_s-\sigma_s^{\frac{2\gamma-1}{2(1-\gamma)}}D_u\sigma_s\right)\right|^2\,ds\notag\\
  \le&2\int_u^t\left|(\phi^{\epsilon}(\sigma^{\epsilon}_s))^{\frac{2\gamma-1}{2(1-\gamma)}}D_u\sigma^{\epsilon}_s-\sigma_s^{\frac{2\gamma-1}{2(1-\gamma)}}D_u\sigma^{\epsilon}_s\right|^2\,ds+2\int_u^t \left|\sigma_s^{\frac{2\gamma-1}{2(1-\gamma)}}D_u\sigma^{\epsilon}_s-\sigma_s^{\frac{2\gamma-1}{2(1-\gamma)}}D_u\sigma_s\right|^2\,ds\notag\\
  \le&2\int_u^t\left|(\phi^{\epsilon}(\sigma^{\epsilon}_s))^{\frac{2\gamma-1}{1-\gamma}}-\sigma_s^{\frac{2\gamma-1}{1-\gamma}}\right|\,ds\sup_{s\in[u,t]}|D_u\sigma^{\epsilon}_s|^2+2\int_u^t \left|D_u\sigma^{\epsilon}_s-D_u\sigma_s\right|^2\,ds\sup_{s\in[u,t]}\left|\sigma_s^{\frac{2\gamma-1}{1-\gamma}}\right|.
  \end{align}
    This converges to $0$ in $L^2(\Omega)$ as well as the first term, so we can conclude that
    \begin{align}
      \int_u^t(\phi^{\epsilon}(\sigma^{\epsilon}_s))^{\frac{2\gamma-1}{2(1-\gamma)}}D_u\sigma^{\epsilon}_s\,dB_s\to\int_u^t\sigma_s^{\frac{2\gamma-1}{2(1-\gamma)}}D_u\sigma^{\epsilon}_s\,dB_s\quad{\rm{in}}\,\,{L}^2(\Omega;H).
    \end{align}
    By Lemma 2.4, we have
    \begin{align}
      D_uX_t=-\frac{1}{2(1-\gamma)}\int_u^t \sigma_t^{\frac{1}{1-\gamma}}D_u\sigma_s\,ds+\rho\sigma_t^{\frac{1}{2(1-\gamma)}}+\frac{1}{2(1-\gamma)}\int_u^t\sigma_s^{\frac{2\gamma-1}{2(1-\gamma)}}D_u\sigma^{\epsilon}_s\,dB_s.
      \end{align}
\end{proof}

\begin{remark*}
  For $\sigma_t$, as well as \rm{Theorem 4.1}, we can more easily prove
  \begin{align}
    D_u{\sigma}_t=&(1-\gamma)\theta\exp\left\{(1-\gamma)\int_u^t \left(-\frac{\kappa\mu\gamma}{1-\gamma}{\sigma_s^{-\frac{1}{1-\gamma}}}+\frac{\gamma{\theta}^2}{2\sigma_s^2}-\kappa\right)\,ds\right\},\\
    \hat{D}_u\sigma_t=&0,
    \end{align}
      for $u \le t$, and $D_u{\sigma}_t=\hat{D}_u\sigma_t=0$\, for \,$u>t$.
\end{remark*}

 \subsection{Malliavin Differentiability of the CEV-Type Heston Model (Actual Price)}

 From now on, we will concentrate on the underlying asset $S_t$ and the volatility $\nu_t$.

 In Subsection 4.4, we proved the Malliavin differentiability of the logarithmic price $X_t$ and the transformed volatility $\sigma_t$. Here we can prove that both of the underlying asset $S_t$ and the volatility $\nu_t$ are Malliavin differentiabile by the chain rule.
  \begin{theorem*}
 $S_t$ and $\nu_t$ belong to ${\bf{D}}^{1,2}$ and we have
    \begin{align}
      D_{t'}S_t=&S_t\left(-\frac{1}{2}\int_{t'}^t D_{t'}\nu_s\,ds+\rho\sqrt{\nu_{t'}}+\frac{1}{2}\int_{t'}^t\nu_s^{-\frac{1}{2}}D_{t'}\nu_s\,dB_s\right),\\
      \hat{D}_{t'}S_t=&S_t\sqrt{1-\rho^2}\sqrt{\nu_{t'}},\\
      D_{t'}\nu_t=&\theta\nu_t^{\gamma}\exp\left\{(1-\gamma)\int_{t'}^t \left(-\frac{\kappa\mu\gamma}{1-\gamma}\nu_s^{-1}+\frac{\gamma{\theta}^2}{2}\nu_s^{-2(1-\gamma)}-\kappa\right)\,ds\right\},\\
      \hat{D}_{t'}\nu_t=&0,
      \end{align}
      for \,${t'}\le{t}$, and $D_{t'}S_t=\hat{D}_{t'}S_t=D_{t'}\nu_t=\hat{D}_{t'}\nu_t=0$\, for\, ${t'}>{t}$.
    \end{theorem*}
 \begin{proof}
  First we consider the Malliavin derivative for $\nu_t$. By Lemma 2.5, we have
  \begin{align}
    D_{t'}\nu_t=&D_{t'}\left(\sigma_t^{\frac{1}{1-\gamma}}\right)={\frac{1}{1-\gamma}}\sigma_t^{\frac{\gamma}{1-\gamma}}D_{t'}\sigma_t={\frac{1}{1-\gamma}}\nu_t^{\gamma}D_{t'}\sigma_t,\\
    \hat{D}_{t'}\nu_t=&{\frac{1}{1-\gamma}}\nu_t^{\gamma}\hat{D}_{t'}\sigma_t.
    \end{align}
    We have by Theorem 4.2
      \begin{align}
        D_{t'}\nu_t=&{\frac{1}{1-\gamma}}\nu_t^{\gamma}(1-\gamma)\theta\exp\left\{(1-\gamma)\int_{t'}^t \left(-\frac{\kappa\mu\gamma}{1-\gamma}{\sigma_s^{-\frac{1}{1-\gamma}}}+\frac{\gamma{\theta}^2}{2\sigma_s^2}-\kappa\right)\,ds\right\}\notag\\
        =&\theta\nu_t^{\gamma}\exp\left\{(1-\gamma)\int_{t'}^t \left(-\frac{\kappa\mu\gamma}{1-\gamma}\nu_s^{-1}+\frac{\gamma{\theta}^2}{2}\nu_s^{-2(1-\gamma)}-\kappa\right)\,ds\right\},\\
        \hat{D}_{t'}\nu_t=&0,
      \end{align}
             for ${t'}\le{t}$, and $D_{t'}\nu_t=\hat{D}_{t'}\nu_t=0$ for \,${t'}>{t}$.
    Next, we consider the Malliavin derivative for $S_t$. By Lemma 2.5, we have
    \begin{align}
    D_{t'}S_t=&D_{t'}e^{X_t}=e^{X_t}D_{t'}X_t=S_tD_{t'}X_t,\\
    \hat{D}_{t'}S_t=&S_t\hat{D}_{t'}X_t.
    \end{align}
    Hence by Theorem 4.2, we have
      \begin{align}
        D_{t'}S_t=&S_t\left(-\frac{1}{2(1-\gamma)}\int_{t'}^t \sigma_t^{\frac{\gamma}{1-\gamma}}D_{t'}\sigma_s\,ds+\rho\sigma_{t'}^{\frac{1}{2(1-\gamma)}}+\frac{1}{2(1-\gamma)}\int_{t'}^t\sigma_s^{\frac{2\gamma-1}{2(1-\gamma)}}D_{t'}\sigma_s\,dB_s\right)\notag\\
        =&S_t\left(-\frac{1}{2}\int_{t'}^t D_{t'}\nu_s\,ds+\rho\sqrt{\nu_{t'}}+\frac{1}{2}\int_{t'}^t\nu_s^{-\frac{1}{2}}D_{t'}\nu_s\,dB_s\right),\\
        \hat{D}_{t'}S_t=&S_t\sqrt{1-\rho^2}{\sigma_{t'}}^{\frac{1}{2(1-\gamma)}}=S_t\sqrt{1-\rho^2}\sqrt{\nu_{t'}},
      \end{align}
 for ${t'}\le{t}$ and $D_{t'}S_t=\hat{D}_{t'}S_t=0$ for ${t'}>t$.
\end{proof}

 \subsection{Delta and Rho}

 Using Theorem 2.4 and Theorem 4.4, we can calculate Greeks of $S_t$. We now consider the following {\SDE}s
 \begin{align}
 dS_t=&rS_t\,dt+\rho\sqrt{\nu_t}S_t{\,dW_t} +\sqrt{1-{\rho}^2}\sqrt{\nu_t}S_td\hat{W}_t,\\
d{\nu}_t =& \kappa(\mu-{\nu}_t)\,dt+ {\theta}{\nu}_t^{\gamma}\,dW_t.
 \end{align}
 Rewrite the {\SDE}s (4.15) and (4.16) as the integral form, and then we have
 \begin{align}
  \left(
  \begin{array}{cc}
    S_t\\
    {\nu}_t
  \end{array}
  \right)=\left(
  \begin{array}{cc}
    {x} \\
    {\nu}_0
  \end{array}
  \right)+\int_0^t \left(
  \begin{array}{cc}
    rS_s \\
    \kappa(\mu-\nu_s)
  \end{array}
  \right)\,ds+\int_0^t \left(
  \begin{array}{cc}
    \sqrt{1-{\gamma}^2}\sqrt{\nu_s}S_s   & \rho\sqrt{\nu_s}S_s  \\
    0 &  \theta\nu_s^{\gamma}
  \end{array}
  \right) \left(
  \begin{array}{cc}
    d\hat{W}s  \\
    \,dW_s
  \end{array}
  \right).
 \end{align}

 We now give the formula for Delta of this model.

 \begin{theorem*}
    Consider the CEV-type Heston model following the dynamics {\rm{(4.15)}} and {\rm{(4.16)}}. We have for any funtion with polynomial growth $\phi:{\bf{R}}\to{\bf{R}}$
  \begin{align}
    \Delta_S=E\left[e^{-rT}\phi(S_T)\int_0^T \frac{1}{xT\sqrt{1-\rho^2}\sqrt{\nu_t}} d\hat{W}_t\right].
  \end{align}
 \end{theorem*}

 \begin{proof}
  Let $\Upsilon$ be the diffusion matrix\,\,$\Upsilon=\left(
  \begin{array}{cc}
    \sqrt{1-{\gamma}^2}\sqrt{\nu_s}S_s   & \rho\sqrt{\nu_s}S_s  \\
    0 &  \theta\nu_s^{\gamma}
  \end{array}
  \right)$, then we can have the inverse\,$\Upsilon^{-1}=\left(
  \begin{array}{cc}
    \frac{1}{\sqrt{1-{\gamma}^2}\sqrt{\nu_s}S_s}   & -\frac{\rho}{\sqrt{1-\gamma^2}{\nu_s}^{\gamma}}  \\
    0 &  \frac{1}{\theta\nu_s^{\gamma}}
  \end{array}
  \right)$. We can have from the It\^o's formula
 \begin{align}
  S_t=x\exp\left\{\int_0^t\left(r-\frac{\nu_s}{2}\right)\,ds+\rho\int_0^t\sqrt{\nu_s}\,dW_s+\sqrt{1-\rho^2}\int_0^t\sqrt{\nu_s}d\hat{W}_s\right\}.
  \end{align}
  Hence we can directly calculate the first variation process $Z_t$ of $\left(
  \begin{array}{cc}
  S_t \\
  {\nu}_t
  \end{array}
  \right)
  $ as $Z_t \coloneqq \frac{\partial}{\partial{x}}\left(
  \begin{array}{ccc}
  S_t  \\
  \nu_t
  \end{array}
  \right)= \left(
  \begin{array}{ccc}
  \frac{S_t}{x} \\
  0
  \end{array}
  \right)$. Then we can have
  \begin{align}
    (\Upsilon^{-1}Z_t)^{\rm{T}}=&Z_t^{\rm{T}}(\Upsilon^{-1})^{\rm{T}}=\left(
    \begin{array}{cc}
      \frac{S_t}{x} &  0\\
    \end{array}
    \right)\left(
    \begin{array}{cc}
      \frac{1}{\sqrt{1-{\rho}^2}\sqrt{\nu_t}S_t}   &  0\\ -\frac{\rho}{\sqrt{1-\rho^2}{\nu^{\gamma}_t}}  &  \frac{1}{\theta\nu_t^{\gamma}}
    \end{array}
    \right)=\left(
    \begin{array}{cc}
      \frac{1}{x\sqrt{1-\rho^2}\sqrt{\nu_t}}&0
    \end{array}
    \right).
  \end{align}
  By Lemma 3.2, we have $\frac{1}{x\sqrt{1-\rho^2}\sqrt{\nu_t}}\in L^4(\Omega\times[0,T])$. As with Theorem 4.3, let $d\dot{W}_t$ be the column with the form $\left(
    \begin{array}{cc}
      d\hat{W}_t\\
      \,dW_t
    \end{array}
    \right)$.
   Since $S_t$ and $\nu_t$ are Malliavin differentiable we have from Theorem 2.4
  \begin{align}
    \Delta_S=&E[\,e^{-rT}\phi(S_T)\frac{1}{T}\int_0^T (\Upsilon^{-1}Z_t)^{\rm{T}}d\dot{W}_t]\notag\\
    =&E\left[e^{-rT}\phi(S_T)\frac{1}{T}\int_0^T \left(
    \begin{array}{cc}
      \frac{1}{x\sqrt{1-\rho^2}\sqrt{\nu_t}} &0
    \end{array}
    \right)
    \left(
    \begin{array}{cc}
      d\hat{W}_t\\
      \,dW_t
    \end{array}
    \right) \right] \notag\\
    =&E\left[e^{-rT}\phi(S_T)\int_0^T \frac{1}{xT\sqrt{1-\rho^2}\sqrt{\nu_t}} d\hat{W}_t\right].
  \end{align}
  \end{proof}

  Moreover we can calculate a Greek, Rho $\varrho$.
  \begin{theorem*}
    Consider the CEV-type Heston model following the dynamics {\rm{(4.15)}} and {\rm{(4.16)}}. Then for any $\phi:{\bf{R}}\to{\bf{R}}$ of polynomial growth, we have
     \begin{align}
      \varrho=E\left[e^{-rT}\phi(S_T)\left(\int_0^T \frac{1}{\sqrt{1-\rho^2}\sqrt{\nu_t}} d\hat{W}_t-T\right)\right].
      \end{align}
    \end{theorem*}

\begin{proof}
  By the definition of $\varrho$, we have
  \begin{align}
    \varrho=\frac{\partial}{\partial{r}}E[\,e^{-rT}\phi(S_T)]=E\left[(-T)e^{-rT}\phi(S_T)+e^{-rT}\phi'(S_T)\frac{\partial{S_T}}{\partial{r}}\right].
  \end{align}
  and $\frac{\partial{S_T}}{\partial{r}}$ as $\frac{\partial{S_T}}{\partial{r}}=\frac{\partial}{\partial{r}}\left(
    \begin{array}{ccc}
      S_T  \\
      \nu_T
    \end{array}
    \right)= \left(
    \begin{array}{ccc}
      T{S_T} \\
      0
    \end{array}
    \right)= xTZ_T$. Here we have
    \begin{align}
    \frac{\partial}{\partial{r}}S_T=&\frac{\partial}{\partial{r}}\left(x\exp\left\{\int_0^T\left(r-\frac{\nu_s}{2}\right)\,ds+\rho\int_0^T\sqrt{\nu_s}\,dW_s+\sqrt{1-\rho^2}\int_0^T\sqrt{\nu_s}d\hat{W}_s\right\}\right)\notag\\
    =&T \cdot S_T.
    \end{align}
By the above formula, we have
\begin{align}
\varrho=&E\left[(-T)e^{-rT}\phi(S_T)+e^{-rT}\phi'(S_T)xTZ_T\right]\notag\\
=&E\left[e^{-rT}\phi(S_T)\left(\int_0^T \frac{1}{\sqrt{1-\rho^2}\sqrt{\nu_t}} d\hat{W}_t-T\right)\right].
\end{align}
\end{proof}
\section{Conclusion}
From Section 3 and 4, it is proved by using unique transformation and approximation that we can apply Malliavin calculus to the CEV model and the CEV-type Heston model both which have non-Lipschitz coefficients in their processes. Then we can provide the formulas to calculate important Greeks as Delta and Rho of these models and contribute to finance, in particular for traders in financial institutions to measure market risks and hedge their portfolios in terms of Delta Hedge.\\
In future it will be required how to calculate the Vega, one of the most important Greeks, for general stochastic volatility models including the CEV-type Heston model. Vega is the sensitivity for volatility but it is difficult to measure Vega for the stochastic volatility models since the volatility is also stochastic process. After the financial crisis, the necessity to grasp the behavior of volatility is increasing. We believe that we can calculate the vega of some important stochastic volatility models such as the Heston model or the CEV-type Heston model by using our results in Section 3 and 4.

\color{black}

\end{document}